\documentclass{siamltex}

\usepackage{amsmath}
\usepackage{amssymb, amsfonts}
\usepackage{graphicx} 
\usepackage{enumitem}

\usepackage[colorlinks=true,bookmarks=false,citecolor=blue,urlcolor=blue]{hyperref} 
\newtheorem{remark}{Remark}[section]

\newtheorem{coro}{Corollary}[section]
\numberwithin{equation}{section}

\newcommand{\bbR}{\mathbb R} 
\newcommand{\bbN}{\mathbb N}

\newcommand{\vv}{\boldsymbol{v}}
\newcommand{\vu}{\boldsymbol{u}}
\newcommand{\vx}{\boldsymbol{x}}

\newcommand{\vn}{\boldsymbol{n}}
\newcommand{\vl}{\boldsymbol{l}}
\newcommand{\valpha}{\boldsymbol{\alpha}}
\newcommand{\vbeta}{\boldsymbol{\beta}}
\newcommand{\vnu}{\boldsymbol{\nu}}
\newcommand{\ve}{\boldsymbol{e}}
\newcommand{\vX}{\boldsymbol{X}}
\newcommand{\vzero}{\boldsymbol{0}}

\newcommand{\nof}[1]{\left|#1\right|}

\begin{document}

\title{Upper bound of high-order derivatives for Wachspress coordinates on polytopes
\thanks{This work was supported by the National Natural Science Foundation of China under Grant No. 12171244.}}


\author{Pengjie Tian%
\thanks{School of Mathematical Sciences, Nanjing Normal University, China.
220901007@njnu.edu.cn}
\and Yanqiu Wang%
\thanks{School of Mathematical Sciences, Nanjing Normal University, China.
yqwang@njnu.edu.cn}
}

\maketitle
\begin{abstract}
The gradient bounds \cite{GilletteRandBajaj2012,RandGilletteBajaj2013,gradientbound2014} of generalized barycentric coordinates play an essential role in the $H^1$ norm approximation error estimate
of generalized barycentric interpolations. Similarly, the $H^k$ norm, $k>1$, estimate needs upper bounds of high-order derivatives,
which are not available in the literature.
In this paper, we derive such upper bounds for the Wachspress generalized barycentric coordinates on simple convex $d$-dimensional polytopes, $d\ge 1$. 
The result can be used to prove optimal convergence for Wachspress-based polytopal finite element approximation of, for example, fourth-order elliptic equations.

Another contribution of this paper is to compare various shape-regularity conditions for simple convex polytopes, and to clarify their relations using knowledge from convex geometry.
\end{abstract}


\section{Introduction} 

Generalized barycentric coordinates (GBCs) \cite{floater2015a, HormannSukumar2017, wachspressnew} have been used to construct various finite element methods (FEMs) on polytopal meshes \cite{GBCFEM-Gillette2010,GBCFEM-Gillette2011,GBCFEM-Martin2008,GBCFEM-Milbradt2008,GBCFEM-Sieger2010,GBCFEM-Sukumar2006,GBCFEM-Sukumar2004,GBCFEM-Tabarraei2006,GBCFEM-Wicke2007}. 
Following the standard FEM theory \cite{brennerscott,ciarlet1978}, FEM approximation error analysis relies on the interpolation error analysis.
For second-order elliptic equations, Gillette, Rand and Bajaj \cite{GilletteRandBajaj2012} have shown that the $H^1$ norm 
GBC interpolation error estimate can be derived using the gradient bounds of GBCs. 
They also gave the gradient bounds of three GBCs: the Wachspress coordinates, the Sibson coordinates and the harmonic coordinates, in two dimensions.
The same authors then proved the gradient bound of the mean value coordinates in \cite{RandGilletteBajaj2013}.
Later, Floater, Gillette and Sukumar \cite{gradientbound2014} re-investigated the Wachspress coordinates, where
they 
extended the gradient bound in \cite{GilletteRandBajaj2012} from $2$D planar polygons to simple convex $d$-dimensional polytopes. 
These results helped to establish the approximation error estimate of GBC-based FEM for second-order elliptic equations.

With the development of high-order GBC-based polygonal elements \cite{RandGilletteBajaj2011a,LaiFloater2016,ChenWang2022}, 
it is possible to construct GBC-based FEM for fourth-order or even higher order elliptic partial differential equations. 
A Morley-type GBC-based non-conforming FEM on polygonal meshes \cite{tian2022} has been developed for solving fourth-order plate bending problems.
Its FEM approximation error analysis relies on the interpolation error analysis, 
which in turn relies on the upper bounds of second-order derivatives of GBCs. 
Such upper bounds were not available by the time when \cite{tian2022} was completed,
which has left a gap in the theoretical error analysis. This motivates the current work.
In this paper, we derive upper bounds of arbitrary order derivatives for Wachspress coordinates on simple convex $d$-dimensional polytopes.
The result can be used to prove optimal convergence for Wachspress-based polytopal finite element approximation of fourth-order elliptic equations.
Moreover, through this work, 
we hope to reveal more properties of the GBCs, which have found interesting applications in computer graphics and computational mechanics \cite{HormannSukumar2017}.

To derive the upper bounds,
one may naturally think of extending the analysis of the gradient bound in \cite{gradientbound2014} to high-order derivatives.
Unfortunately, this turns out to be almost impossible, as the analysis in \cite{gradientbound2014} is tailored to retrieve a sharp gradient bound, and thus highly technical.
We propose different techniques to reach the goal.
Similar to \cite{gradientbound2014}, our upper bounds depend on the diameter $h_K$ of polytope $K$,
and the minimal distance $h_*$ between vertices and non-incident facets ($(d-1)$-dimensional faces) of $K$.
The bounds are sharp in asymptotic orders when $h_*=O(h_K)$, but not when $h_*\ll h_K$.
We shall compare various shape regularity assumptions for simple convex polytopes, and
show that $h_* = O(h_K)$ is a practically reasonable assumption in applications such as FEM.

%
 
The paper is organized as follows. In Section 2, we describe the Wachspress coordinates in detail, and present some preliminary tools. 
In Section 3, we derive the upper bounds of arbitrary-order derivatives for Wachspress coordinates. 
In Section 4, we discuss the relation of geometric conditions for simple convex polytopes. 
Supporting numerical results, including an application on solving fourth-order plate bending problems, are presented in Section 5. 

\section{Wachspress coordinates and some preliminary tools}
Various GBCs have been developed, for example, the Wachspress coordinates \cite{wachspress1975,wachspressnew}, the mean value coordinates \cite{mvc-floater-2003}, 
the Gordon-Wixom coordinates \cite{Gordon-Wixom-1974}, the harmonic coordinates \cite{floater2006}, etc. 
In this paper, we focus on the Wachspress coordinates, which are defined on simple, non-degenerate, convex polytopes in arbitrary dimensions. 
The Wachspress coordinates were first constructed by Wachspress \cite{wachspress1975} on convex polygons in 1975.
Later, Warren et. al. \cite{Warren1996, Warren2007} generalized the construction to arbitrary dimensions, and simplified its formula. 
Other equivalent formulae of the Wachspress coordinates can be found in \cite{wachspress-mvc-floater,floater2015a,floater2006}. 

\subsection{Convex polytope}

We follow Def 0.1 in \cite{Ziegler-1994} to define a $d$-dimensional, $d\ge 1$, convex polytope $K\in \bbR^d$,
which is closed and hence compact.
By ``$d$-dimensional" we mean that the affine hull of $K$ fills the entire $\bbR^d$.
Following Def. 2.1 in \cite{Ziegler-1994}, define {\it faces} of $K$ to be the intersection of $K$ with its supporting hyperplanes.
Again, each face of $K$ is a compact set in $\bbR^d$.
Faces with dimension $0$, $1$ and $d-1$ are called {\it vertices, edges} and {\it facets}, respectively.
Denote by $V$ and $F$, respectively, the set of all vertices and the set of all facets of $K$. When $d=2$, $K$ is a polygon and the facets are just edges of the polygon.

A vertex $\vv$ and a facet $f$ are {\it incident} to each other if $\vv \in f$. Otherwise, we say they are non-incident.
For each facet $f\in F$, denote by $V_f$ the set of all vertices of $K$ incident to $f$. 
For each vertex $\vv\in V$, denote by $V_{\vv}$ the set of all vertices of $K$ connected to $\vv$ through an edge, and by $F_{\vv}$ the set of all facets of $K$ incident to $\vv$. 
We use $|\cdot|$ to denote the cardinality of a given set. For example, $\nof{V}$  is the number of vertices of $K$, and $\nof{F_{\vv}}$ is the number of facets incident to $\vv$. 

 Throughout the entire paper, we assume that the polytope $K$ is
\begin{itemize}
\item {\bf convex};
\item {\bf non-degenerate:} adjacent facets do not lie in the same hyperplane;
\item {\bf simple:} each vertex is incident to exactly $d$ facets, i.e., $\nof{F_{\vv}} = d$ for all $\vv \in V$.
\end{itemize}

\begin{remark}
Convex polytopes defined according to Def 0.1 in \cite{Ziegler-1994} are automatically non-degenerate.
In case that readers may follow other possible definitions of ``convex polytopes" which can be degenerate, we explicitly require polytopes to be non-degenerate in this paper.
In $2$ and $3$ dimensions, polytopes in a centroidal Voronoi tessellation \cite{CVT-Du-1999} are convex and non-degenerate. 
\end{remark}

\begin{remark} According to the properties of simple convex polytopes (see Thm 12.12 in \cite{Arne-book-1983}), we have $\nof{V_{\vv}} = d$ for all $\vv \in V$. 
\end{remark}
  
Let $h_K$ be the diameter of $K$. For each point $\vx\in K$ and facet $f\in F$, denote by $h_f(\vx)$ the distance from $\vx$ to the hyperplane containing $f$.
Define \cite{gradientbound2014}
$$ h_* = \min_{f\in F} \min_{\vv\in V\setminus V_f} h_f(\vv),$$ 
which is the minimum distance between non-incident vertices and facets.
It is clear that $h_*$ is strictly greater than $0$
as long as $K$ is convex and non-degenerate.

\subsection{Wachspress coordinates}
A set of functions $\phi_{\vv}: K \rightarrow \bbR$, for $\vv\in V$, is called a set of {\it generalized barycentric coordinates} (GBCs) on $K$, if it satisfies  \cite{floater2015a}
\begin{enumerate}
	\item (Non-negativity) $\phi_{\vv} (\vx)\geq 0$ on $K$, for all $\vx\in K$;
	\item (Linear precision) for all $\vx \in K$
	\begin{equation} \label{eq:GBCLinearPrecision}
		\sum_{\vv\in V} \phi_{\vv}(\vx) = 1, \qquad\quad \sum_{\vv\in V} \phi_{\vv}(\vx)\, \vv = \vx.
	\end{equation}
\end{enumerate}

When $\nof{V}=d+1$, i.e., $K$ is a simplex, the set of GBCs is uniquely determined by \eqref{eq:GBCLinearPrecision}. 
When $\nof{V} > d+1$, the choice of GBCs is no longer unique. In this paper, we focus on the Wachspress GBCs, and will give a description below.

For each facet $f \in F$, denote by $\vn_f$ the unit outward normal vector on $f$.  Then for any $\vx \in K$, we can express $h_f(\vx)$ by
\begin{equation}\label{eq:hf}
	h_f(\vx) = (\vv-\vx)\cdot\vn_f,\qquad \forall\ \vv \in V_f,
\end{equation}
which is just a linear polynomial.
Function $h_f$ is non-negative in $K$, with $h_f(\vx)=0$ only when $\vx\in f$. 

Given $\vv\in V$, let $f_1,f_2,\ldots,f_d$ be the $d$ facets in $F_{\vv}$, ordered such that the matrix 
\begin{equation} \label{eq:Mv}
 M_{\vv} = [\vn_{f_1},\vn_{f_2},\ldots,\vn_{f_d}] \in \bbR^{d\times d}
 \end{equation}
has a positive determinant
(note that ${\rm det}(M_{\vv})$ is non-zero as long as $K$ is non-degenerate).
The Wachspress coordinates are then defined by
$$
	\phi_{\vv}(\vx) = \frac{w_{\vv}(\vx)}{\sum_{\vu\in V}w_{\vu}(\vx)},\quad \forall \vx\in K,
$$
where 
\begin{equation}\label{golbal-weight-function}
	w_{\vv} = {\rm det}(M_{\vv})\prod_{f\in F\setminus F_{\vv}}h_f(\vx).
\end{equation}
Note that $F\setminus F_{\vv}$ can not be empty, thus \eqref{golbal-weight-function} is well-posed.
For simplicity, denote $W = \sum_{\vu\in V}w_{\vu}$ and consequently we can write
\begin{equation}\label{basic-formula}
	\phi_{\vv}(\vx) = \frac{w_{\vv}(\vx)}{W(\vx)},\quad \forall \vx\in K.
\end{equation}
Clearly, each $w_{\vv}$ is non-negative in $K$. Moreover, we have
\begin{lemma}\label{lem:Wnonzero}
Function $W$ is positive in $K$.
\end{lemma}
\begin{proof}
Since each $w_{\vv}$ is non-negative in $K$ and $W = \sum_{\vv\in V}w_{\vv}$, we only need to show that for each $\vx\in K$, there exists a $\vv\in V$ such that $w_{\vv}(\vx)>0$.
If $\vx$ lies in the interior of $K$,  then by definition $w_{\vv}(\vx)>0$ for all $\vv\in V$. 
Consider the case when $\vx$ lies on the boundary of $K$. 
The intersection of all facets containing $\vx$ is a face with the smallest dimension that contains $\vx$.
Denote this face by $\sigma$.
Then we have $w_{\vv}(\vx)>0$ for all $\vv\in V\cap\sigma$. 
This completes the proof of the lemma.
\end{proof}

By Lemma \ref{lem:Wnonzero}, function $\phi_{\vv}$ is well-defined and has non-negative value in $K$. 
They also satisfy \eqref{eq:GBCLinearPrecision}, and hence form a set of GBCs.

\begin{remark}
The Wachspress coordinates defined in \cite{Warren1996, Warren2007} uses
\begin{equation}\label{golbal-weight-function-alt}
	w_{\vv} = \frac{{\rm det}(M_{\vv})}{\prod_{f\in F_{\vv}}h_f(\vx)},
\end{equation}
instead of \eqref{golbal-weight-function}. It is clear that they are equivalent as $\phi_{\vv}$ remains the same.
The formula \eqref{golbal-weight-function-alt} may be easier to implement, as illustrated in $2$-dimension \cite{floater2015a}.
However, unlike \eqref{golbal-weight-function}, formula \eqref{golbal-weight-function-alt} is not well-defined on $\partial K$.
Besides, its non-polynomial structure has caused much difficulty when estimating the gradient bounds of $\phi_{\vv}$, as seen in \cite{gradientbound2014}.
In this paper, we find it easier to use \eqref{golbal-weight-function} to derive upper bounds of higher order derivatives.
\end{remark}

\begin{remark}
The case $d=1$ is trivial, with $K$ being a line segment.
All notation and results in this paper apply trivially to the case $d=1$. 
Readers shall focus on the case with $d>1$.
\end{remark}

\subsection{Preliminary tools}
 
We introduce the {\it general Leibniz formula} 
for computing partial derivatives of products,
 the {\it multivariate Fa\`{a} di Bruno formula} \cite{FaadiBrunoFormula} for computing partial derivatives of composite functions,
 and some notation.
 
Let $\bbN$ be the set of natural numbers, $\bbN_0$ be the set of non-negative integers, and $\valpha = (\alpha_1,\alpha_2,\ldots,\alpha_d)\in \bbN_0^d$ be a multi-index. Define
\begin{equation*} 
	\begin{aligned}
		|\valpha| &= \sum_{i=1}^{d}\alpha_i,\\
		\valpha ! &= \alpha_1! \alpha_2!\cdots \alpha_d!,\\
		D^{\valpha} &= \frac{\partial^{|\valpha|}}{\partial x_1^{\alpha_1}\cdots\partial x_d^{\alpha_d}}.
	\end{aligned}
\end{equation*}

For $\vbeta = (\beta_1,\beta_2,\ldots,\beta_d) \in \bbN_0^d$, we say $\vbeta\leq\valpha$ if $\beta_i\leq \alpha_i$, for $i =1,2,\ldots,d$. 
Similarly, one can define $<$, $\ge$ and $>$.
When $\vbeta\leq\valpha$, define
$$
	\left(
	\begin{array}{c}
		\valpha \\
		\vbeta
	\end{array}
	\right) = \prod_{i=1}^{d}
	\left(
	\begin{array}{c}
		\alpha_i \\
		\beta_i
	\end{array}
	\right) = \frac{\valpha!}{\vbeta!(\valpha-\vbeta)!}.
$$
We say $\vbeta \prec \valpha$ provided that one of the following holds:
\begin{enumerate}
	\item[(1)] $|\vbeta| < |\valpha|$;
	\item[(2)] $|\vbeta| = |\valpha|$ and $\beta_1< \alpha_1$; 
	\item[(3)] $|\vbeta| = |\valpha|$, $\beta_1 = \alpha_1,\ldots,\beta_k = \alpha_k$ and $ \beta_{k+1} < \alpha_{k+1}$ for some $1\leq k<d.$
\end{enumerate}

{\bf \textit{General Leibniz formula.}} 
For sufficiently smooth functions $f,\, g: \bbR^d\rightarrow \bbR$ and multi-index $\valpha \in \bbN_0^d$, one has
	$$ 
		D^{\valpha}(fg) = \sum_{\vbeta\leq\valpha} 
		\left(
		\begin{array}{c}
			\valpha \\
			\vbeta
		\end{array}
		\right) \left( D^{\valpha-\vbeta}f \right) \left( D^{\vbeta}g\right).
	$$ 

{\bf \textit{Multivariate Fa\`{a} di Bruno formula}  \cite{FaadiBrunoFormula}.} For sufficiently smooth functions $f: \bbR \rightarrow \bbR$, $g: \bbR^d\rightarrow \bbR$ and multi-index $\vbeta\in \bbN_0^d$ with $\vbeta> \vzero$, one has
	\begin{equation}\label{multivariate-Faa-di-Bruno-formula}
		D^{\vbeta}f(g(\vx)) = \sum_{1\leq\lambda\leq|\vbeta|}
		f^{(\lambda)}(g)\sum_{s=1}^{|\vbeta|}\sum_{p_s(\vbeta,\lambda)}\vbeta!\prod_{j=1}^{s}
		\frac{\left(D^{\vl_j}g(\vx)\right)^{k_j}}{k_j!(\vl_j!)^{k_j}},
	\end{equation}
	where 
	\begin{equation}\label{fa-di-bruno-partition}
		\begin{aligned}
			p_s(\vbeta,\lambda) = &\bigg\{ (k_1,\ldots,k_s; \vl_1,\ldots,\vl_s) : k_i \in\bbN \textrm{ and } \vl_i\in \bbN_0^d\textrm{ satisfy}\\
			&\ \textbf{0} \prec \vl_1 \prec \cdots \prec \vl_s, \sum_{i=1}^{s}k_i = \lambda\ and\ \sum_{i=1}^{s}k_i\vl_i = \vbeta \bigg\}.
		\end{aligned}
	\end{equation}
\begin{remark}
	According to \eqref{fa-di-bruno-partition}, the summation in \eqref{multivariate-Faa-di-Bruno-formula} is a finite sum, in which the total number of terms depends only on $\vbeta$.
\end{remark}

Consider the tensor product of vectors $\vx^{(i)} = [x_{1}^{(i)},x_{2}^{(i)},\ldots,x_{d}^{(i)}]^t \in \bbR^d$, for $i=1\ldots k$,
$$ \vX \triangleq \vx^{(1)}\otimes\vx^{(2)}\otimes\cdots\otimes\vx^{(k)} , 
$$ 
with elements $X_{i_1i_2\cdots i_k} = x_{i_1}^{(1)}x_{i_2}^{(2)}\cdots x_{i_k}^{(k)}$. Define the norm
$$\left|\vX\right| = \left(\sum_{i_1,i_2,\ldots, i_k =1}^{d} X_{i_1i_2\ldots i_k}^2  \right)^{\frac{1}{2}}.$$

\begin{remark}
The norm $|\cdot|$ is an extension of the Euclidean length of vectors, which is in turn an extension of the absolute value of scalars.
Thus it is natural to use the same notation $|\cdot|$ for all of them. This is not to be confused with the cardinality of sets such as $\nof{V}$ and $\nof{F_{\vv}}$.
\end{remark}

\begin{lemma}\label{norm-of-tensor-product}
	For $\vx^{(i)} \in \bbR^d$, $i = 1, 2,\ldots, k,$ 
	$$ \left|\vx^{(1)}\otimes\vx^{(2)}\otimes\cdots\otimes\vx^{(k)}\right| = \prod_{i=1}^k |\vx^{(i)}|.$$
\end{lemma}
\begin{proof} 
The case $k=2$ can be calculated directly. 
The rest can be proven by induction.
\end{proof}

Denote by $ \nabla^k$ the $k$th order differential operator, which can be viewed formally as the $k$th order tensor product of the gradient operator $\nabla = [ \frac{\partial}{\partial x_1}, \ldots , \frac{\partial}{\partial x_d} ]^t$, i.e.,
$$ 
	\begin{aligned}
		{\tiny k\; \mathrm{times}\hspace{1cm}} \\[-0.2cm]
		\nabla^k = \overbrace{\nabla \otimes \nabla \otimes \cdots \otimes \nabla } .
	\end{aligned}
$$ 

\section{Derivatives of Wachspress coordinates and the upper bound} 
We first derive the general form of arbitrary derivatives for Wachspress coordinates. By the \textit{general Leibniz formula} and \eqref{basic-formula},  for any $\valpha \in\bbN_0^d$ we have
\begin{equation}\label{D-phi-first-step}
	D^{\valpha}\phi_{\vv} =  D^{\valpha}\left( w_{\vv}\cdot \frac{1}{W} \right)
	= \sum_{\vbeta\leq\valpha} \left(
	\begin{array}{c}
		\valpha \\
		\vbeta
	\end{array}
	\right) \left( D^{\valpha-\vbeta}w_{\vv} \right) \left( D^{\vbeta}\frac{1}{W}\right).
\end{equation}

Then we treat $\frac{1}{W}$ as the composite of an inverse proportional function with $W$, and 
use the \textit{multivariate Fa\`{a} di Bruno formula} to get, for $\vbeta > \vzero$,
\begin{equation}\label{D-phi-second-step}
	\begin{aligned}
		D^{\vbeta}\frac{1}{W} &= \sum_{1\leq\lambda\leq |\vbeta|}(-1)^{\lambda}\lambda!W^{-\lambda-1}\sum_{s=1}^{|\vbeta|}\sum_{p_s(\vbeta,\lambda)}\vbeta!\prod_{j=1}^{s}
		\frac{
			\left(D^{\vl_j}W\right)^{k_j}}{k_j! (\vl_j!)^{k_j}}\\
		&= \sum_{1\leq\lambda\leq |\vbeta|}\sum_{s=1}^{|\vbeta|}\sum_{p_s(\vbeta,\lambda)} \frac{(-1)^{\lambda}\lambda!\vbeta!}{W^{\lambda+1}} \prod_{j=1}^{s}
		\frac{\left(D^{\vl_j}W\right)^{k_j}}{k_j! (\vl_j!)^{k_j}}.
	\end{aligned}
\end{equation}
When $\vbeta=\vzero$, one simply has $D^{\vzero}\frac{1}{W} = \frac{1}{W}$.
Combining \eqref{D-phi-first-step} and \eqref{D-phi-second-step} gives
\begin{equation}\label{formula:D-alpha-phi}
	\begin{aligned}
		& D^{\valpha}\phi_{\vv} \\
		= &\sum_{\vbeta\leq\valpha} \left(
		\begin{array}{c}
			\valpha \\
			\vbeta
		\end{array}
		\right) D^{\valpha-\vbeta}w_{\vv} \sum_{1\leq\lambda\leq |\vbeta|}\sum_{s=1}^{|\vbeta|}\sum_{p_s(\vbeta,\lambda)}\frac{(-1)^{\lambda}\lambda!\vbeta!}{W^{\lambda+1}} \prod_{j=1}^{s}
		\frac{\left(D^{\vl_j}W\right)^{k_j}}{k_j! (\vl_j!)^{k_j}}\\
		= & \frac{D^{\valpha}w_{\vv}}{W}
		+ \sum_{\vzero<\vbeta\leq\valpha} \sum_{1\leq\lambda\leq |\vbeta|}\sum_{s=1}^{|\vbeta|}\sum_{p_s(\vbeta,\lambda)} \left(
		\begin{array}{c}
			\valpha \\
			\vbeta
		\end{array}
		\right)(-1)^{\lambda}\lambda!\vbeta! \frac{D^{\valpha-\vbeta}w_{\vv}}{W^{\lambda+1}} \prod_{j=1}^{s}
		\frac{\left(D^{\vl_j}W\right)^{k_j}}{k_j! (\vl_j!)^{k_j}}.
	\end{aligned} 
\end{equation} 
Next, we estimate the upper bound of $|D^{\valpha}\phi_{\vv}|$. 
For any given $\valpha$, the ($4$-layer) summation in \eqref{formula:D-alpha-phi} is a finite sum with the total number of terms depending only on $\valpha$. 
Hence it suffices to estimate (1) the upper bounds of $|D^{\vnu}w_{\vv}|$ and $|D^{\vnu}W|$, for any multi-index $\vnu \in \bbN_0^d$; 
(2) the lower bound of $W$. According to Lemma \ref{lem:Wnonzero}, we know that $W$ is positive on the compact set $K$.
Therefore its lower bound must be strictly above $0$.

\subsection{Upper bound of $|D^{\vnu}w_{\vv}|$ and $ |D^{\vnu}W|$}
Note that $|D^{\vnu}w_{\vv}|\le |\nabla^k w_{\vv}|$, where $k=|\vnu|$.
We shall estimate $|\nabla^kw_{\vv}|$. 
By \eqref{eq:hf} and \eqref{golbal-weight-function}, the weight function $w_{\vv}$ is a polynomial of degree no greater than $\nof{F}-d$.
Thus we only need to consider the case $k\leq \nof{F}-d.$

\begin{lemma}\label{detM-upper-bound}
Let $K$ be a simple, non-degenerate, convex polytope.
For each $\vv\in V$, the matrix $M_{\vv}$ defined in \eqref{eq:Mv} satisfies
	$$\left(\frac{h_*}{h_K}\right)^d \le {\rm det}(M_{\vv})\leq 1.$$
\end{lemma}
\begin{proof} For $F_{\vv} = \{f_1,f_2,\ldots,f_d\}$, consider a $d$-dimensional parallelepiped $P$ spanned by vectors  $\vn_{f_1},\vn_{f_2},\ldots,\vn_{f_d}$ as follows
$$P(\vn_{f_1},\vn_{f_2},\ldots,\vn_{f_d}) = \left\{ \vx = \sum_{i=1}^{d}t_i\vn_{f_{i}}, 0\leq t_i \leq 1\right\}.$$
In \eqref{eq:Mv} we have assumed that $\{f_1,f_2,\ldots,f_d\}$ is ordered to ensure that $M_{\vv}$ has a positive determinant. Hence $ {\rm det}(M_{\vv})$ is just the volume of $P$.
Since each edge of $P$ has length equal to $1$, we have ${\rm det}(M_{\vv})\leq 1$. 
%
%
%

We then prove the lower bound for $ {\rm det}(M_{\vv})$.
Denote $V_{\vv} = \{\vv_1,\ldots,\vv_d\}$, where $\vv_i$ is the only vertex of $K$ that is adjacent to $\vv$ but non-incident to $f_i$, i.e., $\vv_{i} \notin V_{f_i}$.
Denote by $\ve_i$ be the edge pointing from $\vv$ to $\vv_i$. We conveniently write $\ve_i = \vv_i-\vv$. 
Then the facet $f_i$ lies in the $d-1$ dimensional hyperplane
$$\vv+ {\rm span}\{\ve_1,\ldots,\ve_{i-1},\ve_{i+1},\ldots,\ve_d\}.$$ 
Thus $\vn_{f_i}\cdot \ve_j = 0$, for all $j\neq i.$ 
Define a square matrix $E_{\vv} = [\ve_1,\ve_2,\ldots,\ve_d] \in \bbR^{d\times d}$. 
Then clearly
$$ E_{\vv}^T M_{\vv} = \begin{bmatrix} \ve_1\cdot\vn_{f_1} & 0 & \cdots & 0 \\ 0 & \ve_2\cdot\vn_{f_2} & \cdots & 0 \\ & & \ddots & \\ 0 & 0 & \cdots & \ve_d\cdot\vn_{f_d}\end{bmatrix},$$
which further implies that
\begin{equation}\label{detEdetN-bounds}
 |{\rm det}(E_{\vv})|\ {\rm det}(M_{\vv}) = |{\rm det}(E_{\vv}^T M_{\vv})| = \left|\prod_{i=1}^{d}\ve_i\cdot\vn_{f_i}\right| = \left|\prod_{i=1}^{d}h_{f_i}(\vv_i)\right| \geq h_*^d.
 \end{equation}
Next, let $T_{\vv}$ be the $d$-dimensional parallelepiped spanned by $\ve_1,\ve_2,\ldots,\ve_d$ so that the volume of $T_{\vv}$ is $|{\rm det}(E_{\vv})|$.
Note that each edge of $T_{\vv}$ has length no greater than $h_K$, while the distance from each vertex $\vv_i$, $1\le i\le d$ to a non-incident facet of $T_{\vv}$ has length no less than $h_*$.
By induction one easily gets
\begin{equation}\label{detE-bounds}
	h_K^d \geq |{\rm det}(E_{\vv})| \geq h_*^d.
\end{equation}
Combining \eqref{detEdetN-bounds} and \eqref{detE-bounds} gives the lower bound of ${\rm det}(M_{\vv})$. This completes the proof of the lemma.
\end{proof}

\begin{lemma}\label{bounds-of-nabla-k-wv}
Let $K$ be a simple, non-degenerate, convex polytope.
	For $0\le k\leq \nof{F}-d$, we have
	\begin{equation}
		|\nabla^k w_{\vv}(\vx)| \leq \frac{(\nof{F}-d)!}{(\nof{F}-d-k)!}h_K^{\nof{F}-d-k}, \qquad \forall \vx \in K. 
	\end{equation}
\end{lemma}
\noindent
\begin{proof} Using induction and noticing that $\nabla h_f$ is constant for every facet $f$, we easily get
\begin{equation}\label{formula:nabla-k-wvv}
	\begin{aligned}
		\nabla^k w_{\vv}(\vx)
		=& {\rm det}(M_{\vv}) \, \nabla^k \left( \prod_{f\in F\setminus F_{\vv}}h_f(\vx) \right)\\
		=&{\rm det}(M_{\vv}) \sum_{f_{j_1}\in F\setminus F_{\vv}} \sum_{f_{j_2}\in F\setminus (F_{\vv}\cup \{f_{j_1}\})  } \ldots\sum_{f_{j_k} \in F\setminus (F_{\vv}\cup \{f_{j_1},\ldots,f_{j_{k-1}}\})}\\
		&\nabla h_{f_{j_1}}\otimes\nabla h_{f_{j_2}}\otimes\cdots\otimes\nabla h_{f_{j_k}}\prod_{f_\in F\setminus (F_{\vv}\cup  \{f_{j_1},f_{j_2},\ldots,f_{j_k}\}) }h_{f}(\vx).
	\end{aligned} 
\end{equation}
When $k=\nof{F}-d$, the set $F\setminus (F_{\vv}\cup  \{f_{j_1},f_{j_2},\ldots,f_{j_k}\})$ is empty and we need to formally define 
$\prod_{f\in \emptyset }h_{f} = 1$.
Note that
$\quad |\nabla h_{f_i}(\vx)| = |-\vn_{f_i}| = 1$.
Using \eqref{formula:nabla-k-wvv}, Lemma \ref{detM-upper-bound} and Lemma \ref{norm-of-tensor-product}, we have
\begin{equation*}
	\begin{aligned}
		|\nabla^k w_{\vv}(\vx)| 
		\leq & \sum_{f_{j_1}\in F\setminus F_{\vv}} \ldots\sum_{f_{j_k} \in F\setminus (F_{\vv}\cup \{f_{j_1},\ldots,f_{j_{k-1}}\})} h_K^{\nof{F}-d-k}\\
		= & \frac{(\nof{F}-d)!}{(\nof{F}-d-k)!}h_K^{\nof{F}-d-k}.
	\end{aligned}
\end{equation*}
\end{proof}

\begin{coro}\label{bounds-of-nabla-W}
	For any multi-index  $\vnu \in \bbN_0^d$ such that $|\vnu| = k \le |F|-d$, we have
	$$ 
		|D^{\vnu}W(\vx)|\leq \frac{\nof{V} (\nof{F}-d)!}{(\nof{F}-d-k)!}h_K^{\nof{F}-d-k}, \qquad for\ all\ \vx \in K.
	$$ 
\end{coro}
\begin{proof}
The result follows from $|D^{\vnu}W(\vx)|\leq |\nabla^k W(\vx)|$ and $W=\sum_{\vv\in V} w_{\vv}$.
\end{proof}


 
\subsection{Lower bound of $ W$}
From Lemma \ref{lem:Wnonzero} we know that $W$ must have a lower bound strictly above $0$ in the compact set $K$.
Our next goal is to estimate this lower bound in terms of $h_K$ and $h_*$. 
To this end, we first present two lemmas.
 
\begin{lemma}\label{lemma:d-faces}
Let $K$ be a simple, non-degenerate, convex polytope.
	For any $\vx\in K$, there are at most $d$ facets of $K$, such that the distance from $\vx$ to these facets are strictly less than $\frac{h_*}{d+1}$.
\end{lemma}
\begin{proof} We prove the lemma by contradiction. Fix $\vx\in K$. Assume that there exist $d+1$ facets, denoted by $f_1,f_2,\ldots,f_d,f_{d+1}$, such that 
\begin{equation}\label{d-faces-step-1}
	h_{f_i}(\vx) < \frac{h_*}{d+1},\qquad for\ i = 1,2,\ldots,d+1.
\end{equation}
By the linear precision property of GBCs, we have 
\begin{equation}\label{d-faces-step-3}
	h_{f_i}(\vx) = \sum_{\vv\in V}\phi_{\vv}(\vx)h_{f_i}(\vv) = \sum_{\vv\in V\setminus V_{f_i}}\phi_{\vv}(\vx)h_{f_i}(\vv)\geq \sum_{\vv\in V\setminus V_{f_i}}\phi_{\vv}(\vx)h_*,
\end{equation} 
where in the last step we have used the fact that $\phi_{\vv}$s are non-negative.
Combining inequalities \eqref{d-faces-step-1}-\eqref{d-faces-step-3} and using $\sum_{\vv\in V}\phi_{\vv} = 1$ give
$$\frac{1}{d+1} > \sum_{\vv\in V\setminus V_{f_i}}\phi_{\vv}(\vx) = 1-\sum_{\vv\in V_{f_i}}\phi_{\vv}(\vx).$$
Summing over $i=1\ldots,d+1$ and noticing that every vertex is adjacent to exactly $d$ facets, we have
$$ 
	1>d+1-\sum_{i=1}^{d+1}\sum_{\vv\in V_{f_i}}\phi_{\vv}(\vx)\geq d+1-\sum_{f\in F}\sum_{\vv\in V_{f}}\phi_{\vv}(\vx) = d+1 - d\sum_{\vv\in V}\phi_{\vv}(\vx).
$$ 
Since $\sum_{\vv\in V}\phi_{\vv}(\vx) = 1$, the above inequality yields $1>1$, which is impossible. This completes the proof of the lemma.
\end{proof}

\begin{lemma}\label{common-vertex}
Let $K$ be a simple, non-degenerate, convex polytope.
	Suppose the distance from a point $\vx \in K$ to $k$ facets, $k\ge 1$, are strictly less than $\frac{h_*}{d+1}$, then these $k$ facets have a common vertex.
\end{lemma}
\begin{proof} Denote the $k$ facets by $f_1,f_2,\ldots,f_k$. By Lemma \ref{lemma:d-faces}, we have $k\leq d$. 
We prove the lemma by contradiction. Suppose that these $k$ facets have no common vertex, i.e $\cap_{i=1}^{k}V_{f_i} = \emptyset$. 
Note that \eqref{d-faces-step-3} still holds. Summing \eqref{d-faces-step-3} over $i=1,\ldots,k$ and using the assumption $\cap_{i=1}^kV_{f_i}=\emptyset$, one has
\begin{equation}\label{common-vertex-1}
\begin{aligned}
	\sum_{i=1}^{k}h_{f_i}(\vx) &\geq \sum_{i=1}^{k}\sum_{\vv\in V\setminus V_{f_i}}\phi_{\vv}(\vx)h_*\\
	& \geq h_* \sum_{\vv\in\cup_{i=1}^k\left( V\setminus V_{f_i}\right)}\phi_{\vv}(\vx) 
	&= h_* \sum_{\vv \in V\setminus\left(\cap_{i=1}^kV_{f_i}\right)}\phi_{\vv}(\vx) \\
	&= h_*\sum_{\vv\in V}\phi_{\vv}(\vx) 
	= h_*.
	\end{aligned}
\end{equation}
On the other hand, we also have
\begin{equation}\label{common-vertex-3}
	\sum_{i=1}^{k}h_{f_i}(\vx) < \frac{kh_*}{d+1} < h_*.
\end{equation}
Combine \eqref{common-vertex-1} and \eqref{common-vertex-3} gives $h_*> h_*$, which can not be true.
This completes the proof of the lemma.
\end{proof}
 
Now we are able to derive a lower bound for $W$.
\begin{lemma}\label{W-low-bounds}
Let $K$ be a simple, non-degenerate, convex polytope.
	For $\vx \in K $, one has
	\begin{equation}\label{ine:W-low-bounds}
		W(\vx) \geq \frac{h_*^{\nof{F}}}{\left(d+1\right)^{\nof{F}-d}h_K^d} .
	\end{equation}
\end{lemma}
\begin{proof} For $\vx \in K$, if $h_f(\vx)\ge \frac{h_*}{d+1}$ for all $f\in F$, then by \eqref{golbal-weight-function} and Lemma \ref{detM-upper-bound}
$$
W(\vx) = \sum_{\vv\in V}w_{\vv}(\vx) \ge \sum_{\vv\in V} \left(\frac{h_*}{h_K}\right)^d \left(\frac{h_*}{d+1} \right)^{\nof{F}-d} = \nof{V} \frac{h_*^{\nof{F}}}{\left(d+1\right)^{\nof{F}-d}h_K^d},
$$
and \eqref{ine:W-low-bounds} follows from $|V|\ge 1$.

If there exists exactly $k$ facets, $k\geq 1$, denoted by $f_1,f_2,\ldots,f_k$, such that $h_{f_i}(\vx) < \frac{h_*}{d+1}$, and
$$ 
	h_f(\vx) \geq \frac{h_*}{d+1},\qquad \forall f\in F\setminus \{ f_1,f_2,\ldots,f_k\}.
$$ 
By Lemma \ref{lemma:d-faces},  we know that $k\leq d$. By Lemma \ref{common-vertex}, these $k$ facets have a common vertex, which is denoted by $\vv_0\in\cap_{i=1}^kV_{f_i}$. It is apparent that 
$\{ f_1,f_2,\ldots,f_k \} \subset F_{\vv_0}$.
Therefore
$$ 
	h_f(\vx) \geq \frac{h_*}{d+1},\qquad \forall f\in F\setminus F_{\vv_0}.
$$ 
Combining the above and using Lemma \ref{detM-upper-bound}, we get
$$ 
	W(\vx) \geq w_{\vv_0}(\vx)  \geq \left(\frac{h_*}{h_K}\right)^d  \left(\frac{h_*}{d+1}\right)^{\nof{F}-d}  =\frac{h_*^{\nof{F}}}{ \left(d+1\right)^{\nof{F}-d} h_K^d} .
$$ 
This completes the proof of the lemma.
\end{proof}

\begin{remark}
Prop. 8 in  \cite{GilletteRandBajaj2012} gives a similar lower bound of $W$ for $2$-dimensional convex polygons, where the proof also uses
$W(\vx) \geq w_{\vv_0}(\vx)$.
However, Prop. 8 in \cite{GilletteRandBajaj2012} picked $\vv_0$ as the vertex closest to $\vx$, breaking any tie arbitrarily.
A counter example in Fig. \ref{fig:antiexample} shows that such a choice may only lead to a $0$ lower bound.
In Fig. \ref{fig:antiexample}, $K$ is a quadrilateral with vertices 
$$\vv_1 = (1,\,1),\quad \vv_2 = (0,\,0),\quad \vv_3 = (2,\,0),\quad \vv_4 =(2,1),$$ 
and satisfies all shape regularity conditions in \cite{GilletteRandBajaj2012}.
Draw the line passing through $\vv_1$ and perpendicular to the line $\vv_2\vv_3$.
For any point $\vx$ in this line, the closest vertex is $\vv_1$. However, when $\vx\to E$, one only gets
$$
W(\vx) \geq w_{\vv_1}(\vx) \to 0.
$$

\begin{figure}[h]
	\centering
	\includegraphics[width=2.2in]{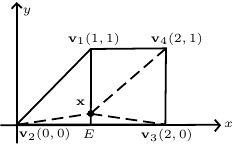}\;
	\caption{{\footnotesize A counter example to the proof of Prop. 8 in \cite{GilletteRandBajaj2012}.}}
	\label{fig:antiexample}
\end{figure}
\end{remark}


\subsection{Main result: the upper bound of $|D^{\valpha}\phi_{\vv}|$}
We are now ready to present the main result on upper bound of arbitrary derivatives of the Wachspress coordinates.
\begin{lemma}\label{bounds-of-phi_v}
Let $K$ be a simple, non-degenerate, convex polytope.
	For $\vx \in K$, one has
	$$ 
		\left|D^{\valpha}\phi_{\vv}\right| \leq C(\nof{F},\nof{V},\valpha,d)  
		 \sum_{\tiny \begin{matrix}0\le\vbeta\leq\valpha\\ |\valpha-\vbeta|\le |F|-d\end{matrix}} \sum_{\textrm{min}\{|\vbeta|,1\}\leq\lambda\leq |\vbeta|} \frac{h_K^{\nof{F}(1+\lambda)-|\valpha|}}{h_*^{\nof{F}(1+\lambda)}} ,
	$$ 
where $C(\nof{F},\nof{V},\valpha,d)$ is a positive constant depending only on $\nof{F}$, $\nof{V}$, $\valpha$ and $d$.
\end{lemma}
\begin{proof}
Using \eqref{formula:D-alpha-phi}, Lemma \ref{bounds-of-nabla-k-wv}, Corollary \ref{bounds-of-nabla-W} and Lemma \ref{W-low-bounds}, we get
\begin{equation*}
	\begin{aligned}
		& |D^{\valpha}\phi_{\vv}|  \\ 
		\leq & C(\nof{F},\nof{V},\valpha,d) \bigg[ \sum_{\tiny \begin{matrix}0<\vbeta\leq\valpha\\ |\valpha-\vbeta|\le |F|-d\end{matrix}} \sum_{1\leq\lambda\leq |\vbeta|}\sum_{s=1}^{|\vbeta|}\sum_{p_s(\vbeta,\lambda)} h_K^{\nof{F}-d-|\valpha-\vbeta|}  \left(\frac{h_*^{\nof{F}}}{h_K^d} \right)^{-(1+\lambda)} \\
		&\qquad\qquad \cdot \prod_{j=1}^{s}
		\frac{1}{k_j! (\vl_j!)^{k_j}}
		  h_K^{(\nof{F}-d-|\vl_j|)k_j}    +  \frac{\textrm{max}\{\nof{F}-|\valpha|, 0\} \, h_K^{\nof{F}-|\valpha|}}{h_*^{\nof{F}}} \bigg] \\
		   = & C(\nof{F},\nof{V},\valpha,d) \bigg[ \sum_{\tiny \begin{matrix}\vbeta\leq\valpha\\ |\valpha-\vbeta|\le |F|-d\end{matrix}} \sum_{1\leq\lambda\leq |\vbeta|}\sum_{s=1}^{|\vbeta|}\sum_{p_s(\vbeta,\lambda)} h_K^{\nof{F}-d-|\valpha-\vbeta|}  \left(\frac{h_*^{\nof{F}}}{h_K^d} \right)^{-(1+\lambda)}\\[2mm]
		&\qquad\qquad \cdot 
		  h_K^{(\nof{F}-d)\lambda- |\vbeta|} +  \frac{\textrm{max}\{\nof{F}-|\valpha|, 0\} \, h_K^{\nof{F}-|\valpha|}}{h_*^{\nof{F}}} \bigg],
		  \end{aligned}
\end{equation*}
where in the last step we have used $(\nof{F}-d)(k_1+\ldots+k_s)-\sum_{j=1}^{s}k_j|\vl_j| = (\nof{F}-d)\lambda- |\vbeta|$.
Notice that all summations in the above are finite sums and $|\valpha-\vbeta|=|\valpha|-|\vbeta|$ when $\vbeta\le \valpha$, we further get
$$
\begin{aligned}
		|D^{\valpha}\phi_{\vv}| &\le  C(\nof{F},\nof{V},\valpha,d) 
		\bigg[ \sum_{\tiny \begin{matrix}0< \vbeta\leq\valpha\\ |\valpha-\vbeta|\le |F|-d\end{matrix}} \sum_{1\leq\lambda\leq |\vbeta|} \frac{h_K^{\nof{F}(1+\lambda)-|\valpha|}}{h_*^{\nof{F}(1+\lambda)}} \\
		&\qquad\qquad +  \frac{\textrm{max}\{\nof{F}-|\valpha|, 0\} \, h_K^{\nof{F}-|\valpha|}}{h_*^{\nof{F}}} \bigg].
		\end{aligned}
$$
This completes the proof of the lemma.
\end{proof}

Lemma \ref{bounds-of-phi_v} immediately implies that:
\begin{theorem} \label{cor:bound}
	When $\nof{F}$, $\nof{V}$ are bounded and $h_* = O(h_K)$, we have
	$$ \left|D^{\valpha}\phi_{\vv}\right| \leq C h_K^{-|\valpha|}, $$
	where $C$ depends only on $\nof{F}$, $\nof{V}$, $\valpha$ and $d$.
\end{theorem}

\begin{remark}
One may wonder if the result in Theorem \ref{cor:bound} can be simply proved through a scaling argument instead of the long process presented in this section.
We point out that there is no affine scaling for general polytopes. Though one can still use a uniform scaling to map $K$ into a ``reference polytope" with characteristic size of $O(1)$, 
and then perform a uniform scaling argument to derive $\left|D^{\valpha}\phi_{\vv}\right| \leq C h_K^{-|\valpha|}$. However in this case the constant $C$ will depend on the shape of $K$.
\end{remark}

\begin{remark}
When $h_* = O(h_K)$, the upper bound in Theorem \ref{cor:bound} is obviously sharp, in the asymptotic order, according to a uniform scaling argument.
Unfortunately, when $h_*\ll h_K$, the upper bound in Lemma \ref{bounds-of-phi_v} is far from sharp, as can be easily seen by examining $d$-dimensional simplices.

When $|\valpha|=1$, a sharp (in asymptotic order) upper bound $\left|D^{\valpha}\phi_{\vv}\right|\le C h_*^{-1}$ has been proved in \cite{gradientbound2014}.
However, the proof in \cite{gradientbound2014} is highly technique and thus can hardly be generalized to cases with $|\valpha|>1$.

The result in this section is mostly useful in the scenario when $h_* = O(h_K)$ and higher order derivative bounds are needed.
In Section 5, we shall further explore this through numerical experiments.
\end{remark}

\section{Geometric assumptions}\label{geometric-assumption}
According to Theorem \ref{cor:bound}, for given dimension $d$ and multi-index $\valpha$, 
the upper bound of the derivatives for Wachspress coordinates can be characterized solely by $\nof{F}$, $\nof{V}$, $h_K$ and $h_*$.
In this section, we discuss conditions on $K$ that ensure these characteristic quantities be bounded.

Let $\rho_K$ be the radius of the largest inscribed ball in $K$ and $R_K$ be the radius of the smallest ball containing $K$ 
(not to be confused with the circumsphere of simplices which must pass through all vertices).
Let $w_K$ be the width of $K$, i.e. the minimal distance of two parallel supporting hyperplanes of $K$. 
Here by ``supporting hyperplane" we mean a $d-1$ dimensional hyperplane whose intersection with $K$ is a face of $K$. 
An illustration in $2$-dimension is given in Figure \ref{fig:width-of-k}.

The relation of $\rho_K$, $R_K$, $w_K$ and $h_K$ has been studied thoroughly by the convex geometry community.
First, it is obvious that 
$$
2\rho_K\le w_K\le h_K\le 2 R_K.
$$
The upper bounds of $R_K/h_K$ and $w_K/\rho_K$ are given by 
the Jung's theorem and the Steinhagen's Theorem, respectively \cite{Betke-Henk-1993,DanzerGruenbaumKlee1963,Eggleston-1958,Grunbaum2003,yaglom1961}. 

{\bf \textit{Jung's theorem}} (see (2.6) in \cite{DanzerGruenbaumKlee1963}). Let $K$ be a simple, non-degenerate, convex polytope in $\bbR^d$, then $h_K$ and $R_K$ satisfy
\begin{equation*}
	h_K \geq R_K \sqrt{\frac{2(d+1)}{d}}.
\end{equation*}

{\bf \textit{Steinhagen's theorem}} (see (1.1) in \cite{Betke-Henk-1993}). Let $K$ be a simple, non-degenerate, convex polytope in $\bbR^d$, then $w_K$ and $\rho_K$ satisfy
$$
  \begin{aligned}
     w_K &\le \rho_K (2d^{\frac{1}{2}}),\qquad &&d\ \textrm{odd},\\
     w_K &\le 2(d+1) \rho_K/(d+2)^{\frac{1}{2}},\qquad &&d\ \textrm{even}.
  \end{aligned}
$$ 

\begin{figure}[h]
	\centering \includegraphics[width=6cm]{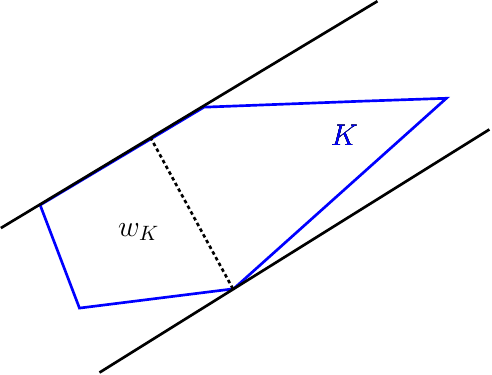}\;
	\caption{{\footnotesize A convex polygon is located between two parallel lines (supporting hyperplanes). The minimum distance between all pairs of parallel supporting hyperplanes is the width of the polygon.}}
	\label{fig:width-of-k}
\end{figure}

Combining the above, one gets the following proposition.
\begin{proposition} \label{prop:geometric}
On $d$-dimensional simple, non-degenerate, convex polytopes, one has $h_K\approx R_K$ and $w_K\approx \rho_K$.
That is, there exist positive constants $C_i$, $i=1,2,3,4$, depending only on $d$ such that 
$$
C_1 R_K \le h_K\le C_2 R_K,\qquad C_3\rho_K\le w_k\le C_4\rho_K.
$$
\end{proposition}

We then list some commonly used geometric assumptions on the shape of the polytope $K$ and study their relations.
Constants in the following assumptions may depend on $d$ but not on the shape of $K$. 
For readers not familiar with the finite element language, we emphasize that these constants are not meant for one single polytope,
but for multiple polytopes with various shapes.
\begin{enumerate}
	\item[\textbf{(H1)}]There exists a positive constant $C_* \in \bbR$ such that $h_* \geq C_* h_K$.
	\item[\textbf{(H2)}] (Minimum edge length) There exist a positive constant $D_* \in \bbR$ such that $|\vu-\vv|\geq D_* h_K$ for any $\vu,\vv \in V$, $\vu\neq\vv$.
	\item[\textbf{(H3)}](Maximum vertex count) There exists a positive integer $N_V^*$ such that $ \nof{V} \leq N_V^*.$
	\item[\textbf{(H4)}](Maximum facet count) There exists a positive integer $N_F^*$ such that $ \nof{F} \leq N_F^*.$
    \item[\textbf{(H5)}]There exists a positive constant $C \in \bbR $ such that $ C h_K \le w_K$.
\end{enumerate}

According to Proposition \ref{prop:geometric}, it is not hard to see that Assumption \textbf{(H5)} is equivalent to

\begin{enumerate}
\item[\textbf{(H5')}] (Chunkiness parameter) There exists a positive constant $\gamma^*$, such that 
  $$\frac{h_K}{\rho_K} \le \gamma^*.$$
\end{enumerate}

\begin{remark} 
For convex polytopes which are always star-shaped with respect to its inscribed ball, Assumption \textbf{(H5')} is just the well-known {\bf chunkiness parameter assumption}
in the finite element literature (see definition 4.2.16 in \cite{brennerscott}).
\end{remark}

We now state an important observation.
\begin{lemma}\label{relationship-of-geometric-assumption}
	Let $K \subset \bbR^d$ be a non-degenerate simple convex polytope, then 
	$$\textbf{(H1)} \Rightarrow \textbf{(H2)} \Rightarrow \textbf{(H3)} \Leftrightarrow \textbf{(H4)}\quad and\quad  \textbf{(H1)} \Rightarrow \textbf{(H5)}\Leftrightarrow \textbf{(H5')}.$$
\end{lemma}
\begin{proof}
	We first prove \textbf{(H1)} $\Rightarrow$ \textbf{(H2)}. It is clear that for any facet $f\in F_{\vv}$ that does not contain $\vu$ (such an $f$ does exist), one has
	\begin{equation*}
		|\vu-\vv|\geq |(\vu-\vv)\cdot\vn_f| = h_f(\vu) \geq h_* \ge C_* h_K.
	\end{equation*}
		
	\textbf{(H2)} $\Rightarrow$ \textbf{(H3)} has been proved, for $2$-dimension, in Proposition 4 of \cite{GilletteRandBajaj2012}. 
	It is straight forward to extend to arbitrary dimensions. For reader's convenience, we present it below.
	Let $\{B(\vv, \frac{D_*h_K}{2})\}_{\vv \in V}$ be the set of open balls centered at each vertex
	$\vv\in V$ and with radius $\frac{D_*h_K}{2}$. By Assumption \textbf{(H2)}, the balls are disjoint.
	By the definition of $R_K$, there exists a point $\vx\in K$ such that $K\subset \overline{B(\vx,R_K)}$. One has
	$$ \bigcup_{\vv\in V}B(\vv, \frac{D_*h_K}{2}) \subset B(\vx,R_K+\frac{D_*h_K}{2}). $$
	Comparing the volume of the balls gives
	$$ \nof{V} \left(\frac{D_*h_K}{2}\right)^d \leq \left(R_K+\frac{D_*h_K}{2}\right)^d.  $$ 
	Using the Jung's theorem, we arrive at
	$$ \nof{V} \leq
	\left(\frac{\sqrt{\frac{2d}{d+1}} +D_*}{D_*}\right)^d. $$
	
	Next, we prove \textbf{(H3)} $\Leftrightarrow$ \textbf{(H4)}. For $d$-dimensional simple polytopes, one clearly has $\nof{F}\le d\nof{V}$. Thus \textbf{(H3)} $\Rightarrow$ \textbf{(H4)}.
	By setting $j=0$ in Theorem 18.1 of \cite{Arne-book-1983}, one gets \textbf{(H3)} $\Leftarrow$ \textbf{(H4)}.

    Finally, we prove \textbf{(H1)} $\Rightarrow$ \textbf{(H5)}. 
    By \eqref{detE-bounds}, we know that the volume of $K$ satisfies
    \begin{equation}\label{ine:volume-K-low-bound}
    vol(K) \ge \frac{1}{d!} |{\rm det}(E_{\vv})| \ge \frac{1}{d!} h_*^d \ge \frac{C_*^d}{d!} h_K^d,
    \end{equation}
    where the last step follows from Assumption \textbf{(H1)}.
By the definition of $h_K$ and $w_K$, one can always fit a polytope $K$ in a hypercube with size $[0,h_K]^{d-1}\times [0,w_K]$. Hence
    \begin{equation}\label{ine:volume-K-up-bound}
      vol(K) \le w_Kh_K^{d-1}.
    \end{equation}
%
%
%
%
    Combining \eqref{ine:volume-K-low-bound}-\eqref{ine:volume-K-up-bound} gives Assumption \textbf{(H5)}.

Since \textbf{(H5)} is equivalent to \textbf{(H5')}, this completes the proof of the lemma.
\end{proof} 



\begin{remark}
   Assumptions \textbf{(H5)} and \textbf{(H5')} exclude ``thin" polytopes, while Assumption \textbf{(H1)} is stronger. 
   An illustration is given in Fig.\ref{fig:geo-counterexample}.
   \begin{figure}[h]
  \centering
  \includegraphics[width=4cm]{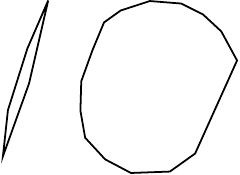}\;
  \caption{The left polygon does not satisfy Assumption \textbf{(H5)} or \textbf{(H5')}. The right polygon satisfies Assumptions \textbf{(H5)} and \textbf{(H5')}, 
  but not \textbf{(H1)}.}\label{fig:geo-counterexample}
\end{figure}
\end{remark}

When $K$ is a non-degenerate, convex polygon in $\bbR^2$, more geometric assumptions can be considered:
\begin{enumerate}
	\item[\textbf{(H6)}](Maximum interior angle) There exists $\alpha^*<\pi $ such that all interior angles of $K$ are bounded above by $\alpha^*$.
	\item[\textbf{(H7)}](Minimum interior angle) There exists $ \alpha_*>0$ such that all interior angles of $K$ are bounded below by $\alpha_*$.
\end{enumerate}  

\begin{lemma} \label{lem:2D}
  Let $K$ be a non-degenerate, convex polygon in $\bbR^2$. In addition to the relations presented in Lemma \ref{relationship-of-geometric-assumption}, one has
  $$
  \textbf{(H6)} \Rightarrow \textbf{(H3)},\qquad  \textbf{(H5')} \Rightarrow \textbf{(H7)},
  $$
  and 
  $$
  \textbf{(H1)} \Leftrightarrow  \bigg(\textbf{(H2)} + \textbf{(H6)} + \textbf{(H7)}\bigg).
  $$
\end{lemma}
\begin{proof}
  \textbf{(H6)} $\Rightarrow$ \textbf{(H3)}  and \textbf{(H5')} $\Rightarrow$ \textbf{(H7)} have been proved in \cite{GilletteRandBajaj2012}.
   It has been proved in Corollary 2.4 of \cite{gradientbound2014} that
$$ \bigg(\textbf{(H2)} + \textbf{(H6)} + \textbf{(H7)}\bigg) \Rightarrow \textbf{(H1)}.$$
Combining the above with Lemma \ref{relationship-of-geometric-assumption}, one only needs to prove \textbf{(H1)} $\Rightarrow$ \textbf{(H6)}.

Denote by $\vv_1,\vv_2,\ldots,\vv_n$ the vertices of $K$, ordered counter-clockwisely. Let $f_i$ be the edge pointing from veterx $\vv_i$ to vertex $\vv_{i+1}$. 
Let $\theta_i$ be the interior angle at vertex $\vv_i$. Since $K$ is non-degenerate and convex, we have $\theta_i <\pi$.
Moreover,
	$$ h_K \sin\theta_i \geq |\vv_{i+1} - \vv_{i}|\sin\theta_i = h_{f_{i-1}}(\vv_{i+1}) \geq h_* \geq C_* h_K.$$
Hence $\sin\theta_i \ge C_*>0$, which imples both \textbf{(H6)} and \textbf{(H7)}.
\end{proof}

From Lemma \ref{relationship-of-geometric-assumption}, we know that Assumption \textbf{(H1)} ensures that $h_*=O(h_K)$ and $\nof{F}$, $\nof{V}$ being bounded,
which implies that Theorem \ref{cor:bound} holds.

\section{Numerical results}  
We only performed numerical experiments in $2$D, not only because $3$-dimensional experiments are too demanding on time and computer hardwares,
but also because carefully designed $2$-dimensional experiments have already yielded rich and interesting results adequate to support the goal of this paper.

Numerical experiments are designed to clarify four issues: (1) Can one easily obtain high quality polygonal meshes with $h_* = O(h_K)$? 
(2) When $h_* = O(h_K)$, do numerical results agree well with Theorem \ref{cor:bound}? 
(3) What happens when $h_* \ll h_K$? 
(4) When does one need a higher ($|\valpha|>1$) derivative bound?
This section is consequently organized in the above order.

For a given non-degenerate, convex polygon $K$, for simplicity, denote by $n=|F|=|V|$ the number of vertices of $K$, and by $\phi_i$, $1\le i\le n$, the Wachspress
coordinate associated with the $i$th vertex of $K$. Define
$$ \Lambda_{\valpha} = \max_{\vx\in K}\sum_{i=1}^{n}\left|D^{\valpha}\phi_i(\vx)\right|.$$
To facilitate the experiments, we adopt the method in \cite{ChenWang2020,HuangWang2020} to generate random convex polygons 
by calculating the convex hull of $20$ random points in $[0, 1]\times[0, 1]$. The polygons can then be uniformly scaled to any size $h_K$ as needed.

In all experiments except for the FEM one in subsection \ref{subsec:FEM}, $D^{\valpha}\phi_i(\vx)$ is computed symbolically and extensively tested to ensure its high accuracy.

\begin{figure}[h]
	\centering
	\includegraphics[width=6cm]{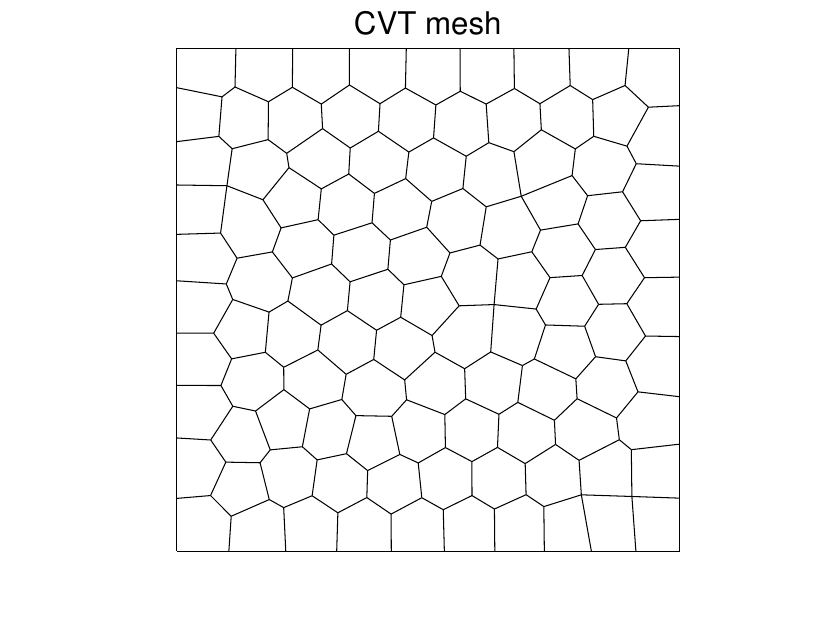}\;
	\includegraphics[width=6cm]{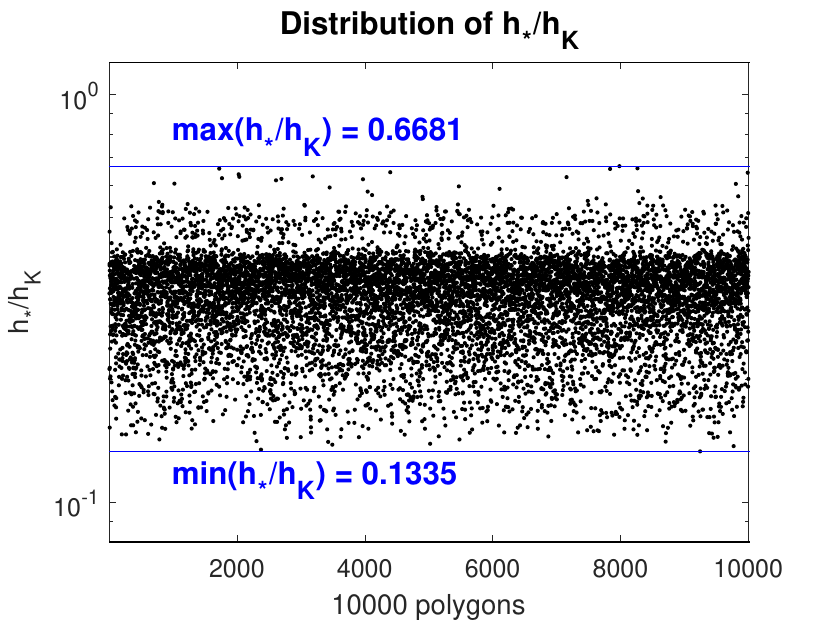}\;
	\includegraphics[width=6cm]{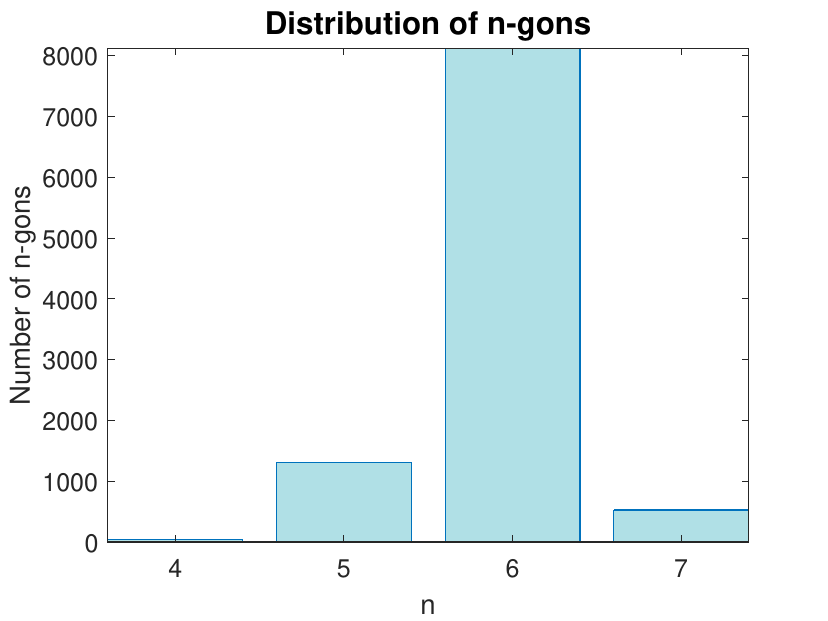}\;  
	\includegraphics[width=6cm]{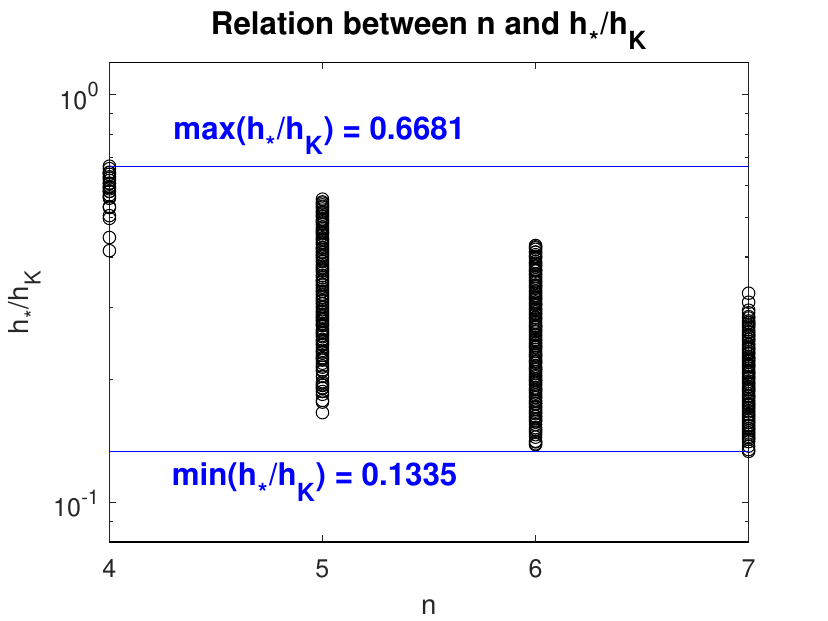}\;
	\caption{Statistics of polygons in a CVT mesh after eliminating ``short" edges. 
		(1) A CVT mesh after eliminating ``short" edges. (2)-(3) The distributions of $h_*/h_K$ and $n$-gons in a CVT mesh containing $10000$ polygons. 
		(4) $\frac{h_*}{h_K}$ vs. $n$ for the same CVT mesh.}\label{fig:CVT-polygon}
\end{figure}

\subsection{Can one easily obtain high quality polygonal meshes with $h_* = O(h_K)$? }
The answer is definitely yes. Consider the centroidal Voronoi tessellations (CVTs) \cite{CVT-Du-1999} (see Fig. \ref{fig:CVT-polygon}).
All polygons in a CVT mesh are non-degenerate and convex.
By definition, it is not hard to prove that a quasi-uniform CVT polygonal mesh must satisfy maximum interior angle condition \textbf{(H6)}
and the minimum interior angle condition \textbf{(H7)}. Then, according to Lemma \ref{lem:2D}, condition \textbf{(H1)} holds as long as condition \textbf{(H2)} holds,
i.e., $h_* = O(h_K)$ as long as there is no ``short" edge in the mesh. There is no built-in mechanism in CVT to prevent the existence of ``short" edges.
However, we can easily ``eliminate" the ``short" edges from a CVT, by simply shrink a ``short" edge into a vertex.

In practice, we first generate a CVT mesh with characteristic mesh size $h$ using the code {\tt polymesher} written by Talischi et al. \cite{polymesher}, 
and then check each edge in the mesh. If the edge has length less than $0.15h$, we simply remove one end point of the edge, so that the ``short" edge vanishes.
The result turns out to be satisfying, as one can see from Fig. \ref{fig:CVT-polygon} that $\frac{h_*}{h_K}$ is now controlled in a reasonable range of $[0.1335,\,0.6681]$.
Another interesting fact about CVT, which can be observed in Fig. \ref{fig:CVT-polygon}, is that it appears to consist of mostly hexagons.

As a comparison, we also examine $\frac{h_*}{h_K}$ and the distribution of $n$-gons of $10000$ random convex polygons. ``Short" edges are not eliminated this time. 
In a test run presented in Fig. \ref{fig:rand-polygon}, the majority of polygons has $\frac{h_*}{h_K} \in [0.0024,\, 0.6810]$, 
with only one polygon having $\frac{h_*}{h_K} = 6.3552e-04$. 
The distribution of $n$-gons appears to be  a normal distribution with respect to $n$. 

\begin{figure}[h]
  \centering
  \includegraphics[width=6cm]{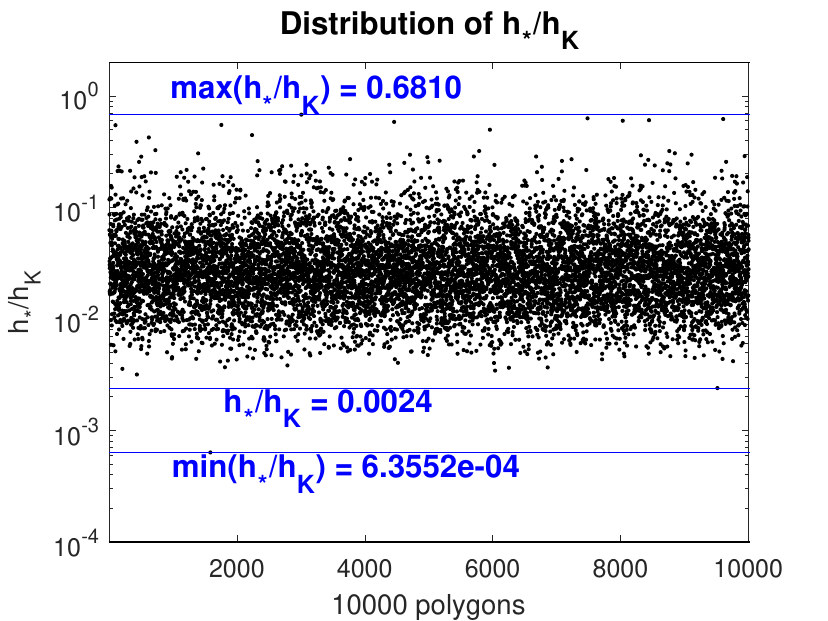}\;
  \includegraphics[width=6cm]{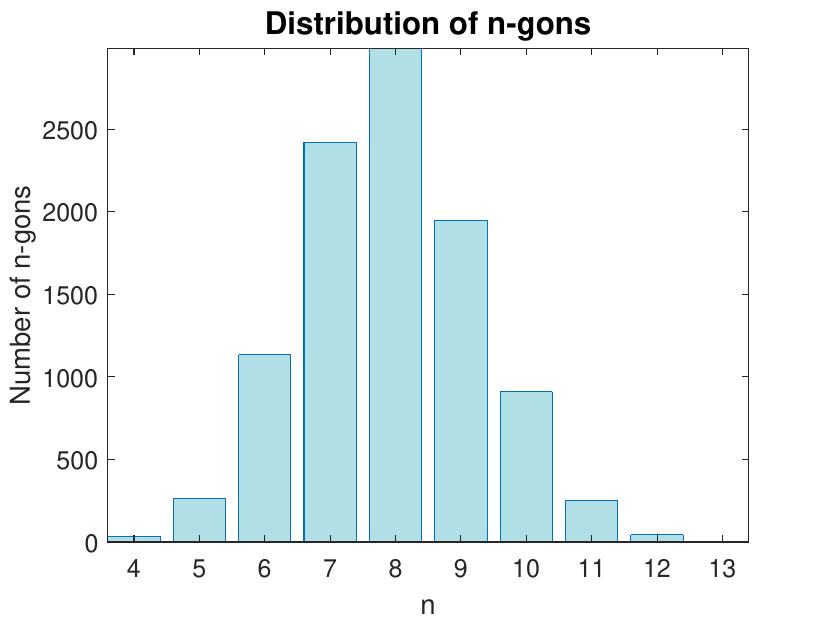}\;  
  \includegraphics[width=6cm]{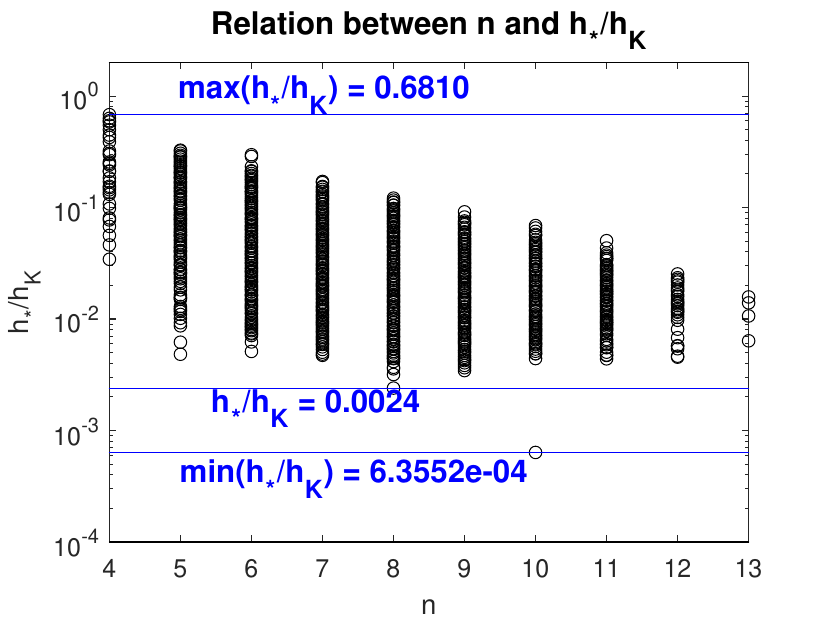}\;
  \caption{Statistics of random convex polygons. (1)-(2) The distributions of $h_*/h_K$ and $n$-gons in $10000$ random convex polygons.
   (3) $\frac{h_*}{h_K}$ vs. $n$ for the same $10000$ random convex polygons.}\label{fig:rand-polygon}
\end{figure}

In Figs. \ref{fig:CVT-polygon}-\ref{fig:rand-polygon}, we also report $\frac{h_*}{h_K}$ vs. $n$. In both cases, the upper bound of $\frac{h_*}{h_K}$ appears to
decay exponentially as $n$ grows.

From Figs. \ref{fig:CVT-polygon}-\ref{fig:rand-polygon}, we draw the conclusion that CVT meshes, after eliminating ``short" edges, 
are high quality polygonal meshes with $h_* = O(h_K)$.
An alternative strategy to remove ``short" edges from a CVT by using the duality between Voronoi and Delaunay meshes
can be found in \cite{GBCFEM-Sieger2010}.

\subsection{When $h_* = O(h_K)$, do numerical results agree well with Theorem \ref{cor:bound}}
To examine this, we generate $100$ random convex polygons with size $h_K$ randomly distributed in $[10^{-5},\, 1]$, and $\frac{h_*}{h_K}\ge 0.01$.
This is achieved by generating random convex polygons one by one, and throwing away polygons with $\frac{h_*}{h_K}< 0.01$ until we have $100$ qualified ones left.
These $100$ polygons can roughly be considered as having $h_* = O(h_K)$. We also point out that all polygons in the CVT mesh shown in Fig. \ref{fig:CVT-polygon}
satisfy $\frac{h_*}{h_K}\ge 0.01$.

We then compute $\Lambda_{\alpha}$, for $|\valpha| = 1,2,3$, on these $100$ random polygons. The results are reported in Figs. \ref{fig:example1-alpha=1}-\ref{fig:example1-alpha=3}. 
In each graph, it can be observed that the values of $\Lambda_{\alpha}$ are clustered around a line $\Lambda_{\alpha} = Ch_K^{-|\valpha|}$, 
which agrees well with the theoretical prediction of Theorem \ref{cor:bound}.

\begin{figure}[h]
    \begin{center}
          \includegraphics[width=6cm]{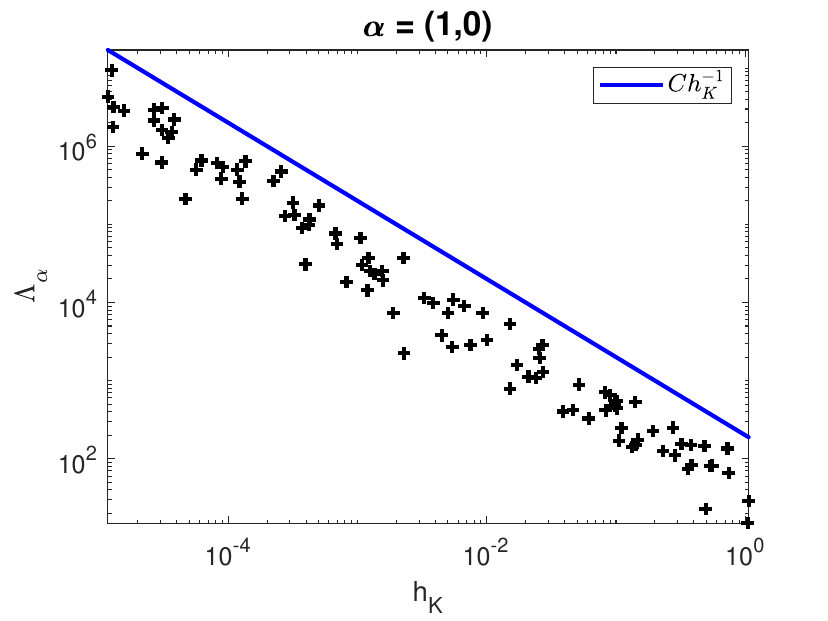}\;
          \includegraphics[width=6cm]{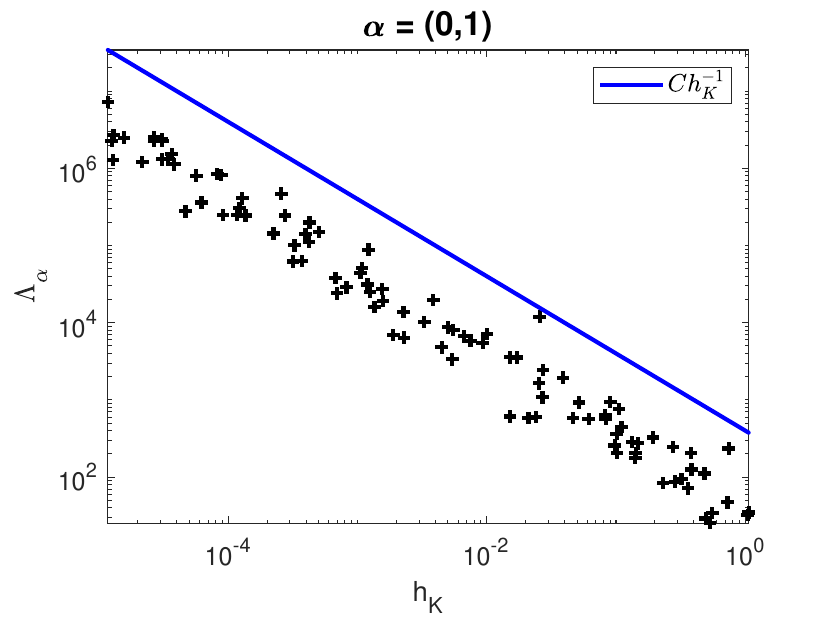}\;
          \caption{Relation of $\Lambda_{\valpha}$ and $h_K$ with $|\valpha| = 1$ on random polygons with $\frac{h_*}{h_K} \geq 0.01$. }\label{fig:example1-alpha=1}
    \end{center}
\end{figure}

\begin{figure}[h]
    \begin{center}
          \includegraphics[width=6cm]{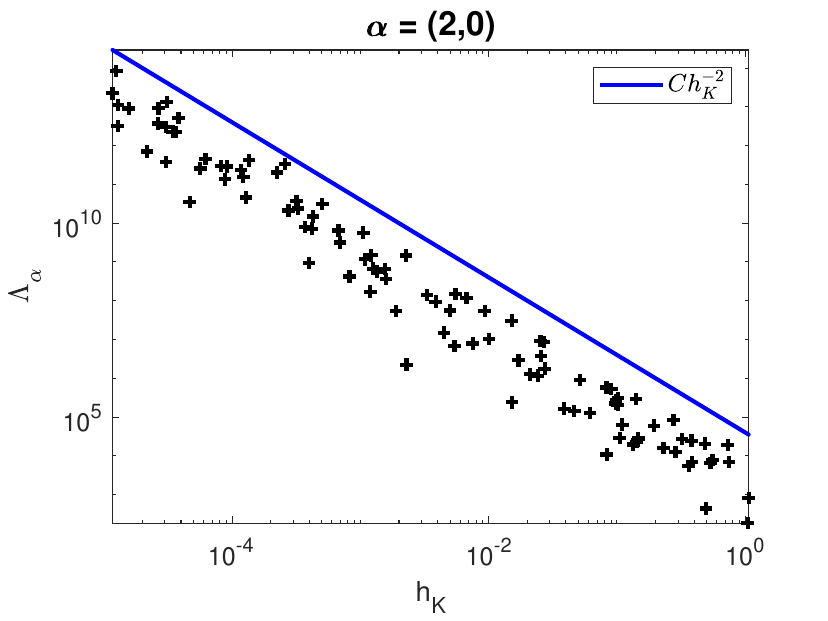}\;
          \includegraphics[width=6cm]{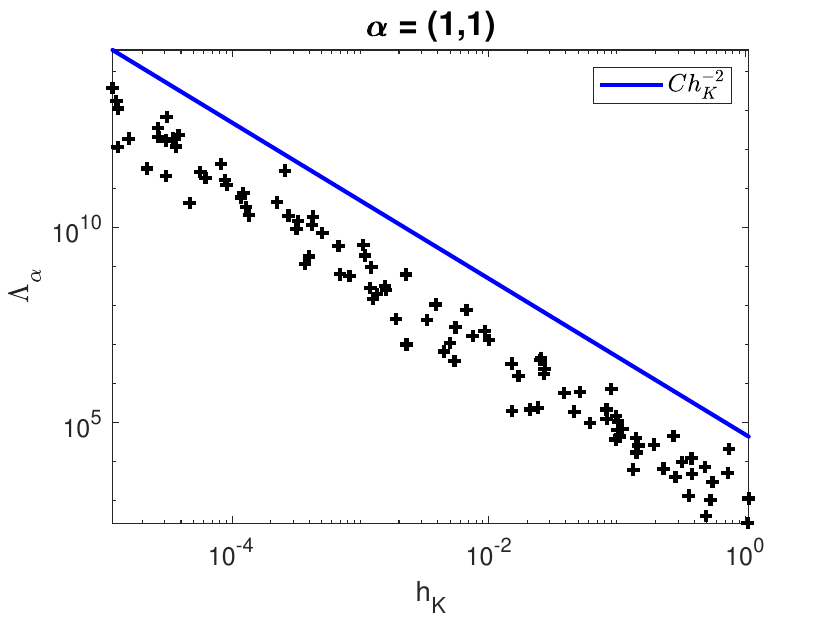}\;
          \includegraphics[width=6cm]{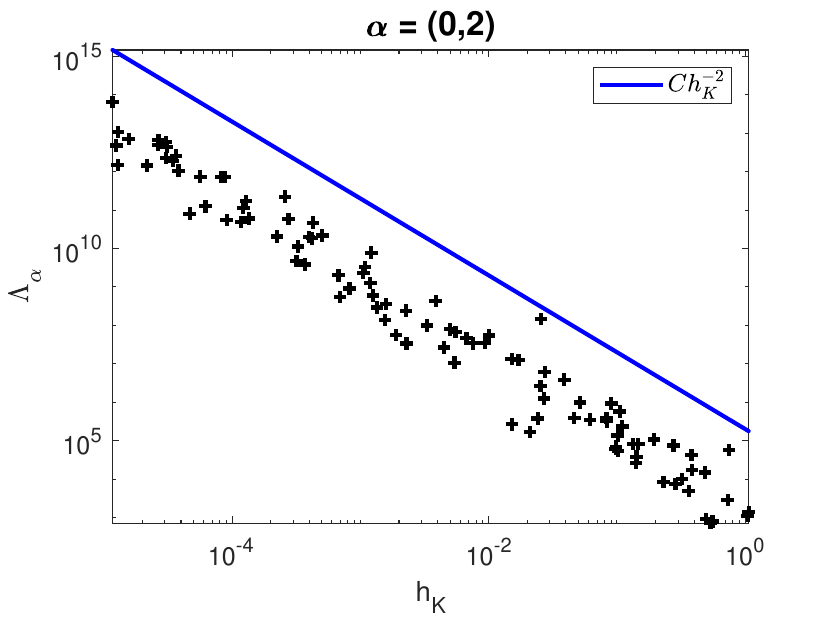}\;
          \caption{Relation of $\Lambda_{\valpha}$ and $h_K$ with $|\valpha| = 2$ on random polygons with $\frac{h_*}{h_K} \geq 0.01$. }\label{fig:example1-alpha=2}
    \end{center}
\end{figure}

\begin{figure}[h]
    \begin{center}
          \includegraphics[width=6cm]{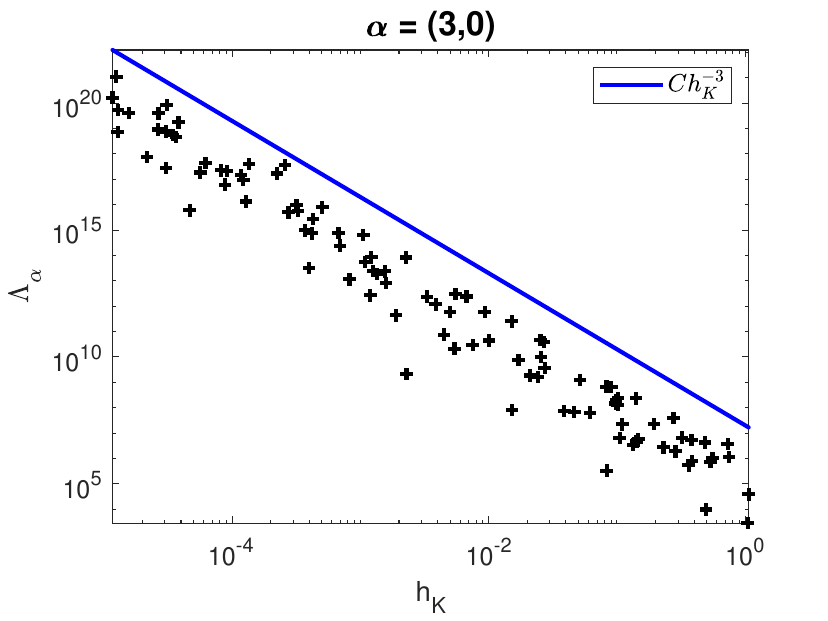}\;
          \includegraphics[width=6cm]{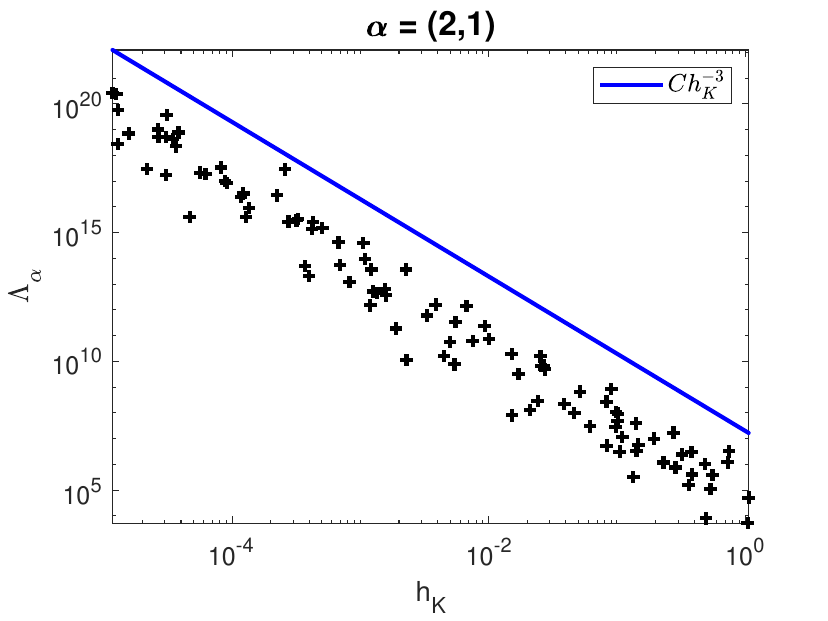}\;
          \includegraphics[width=6cm]{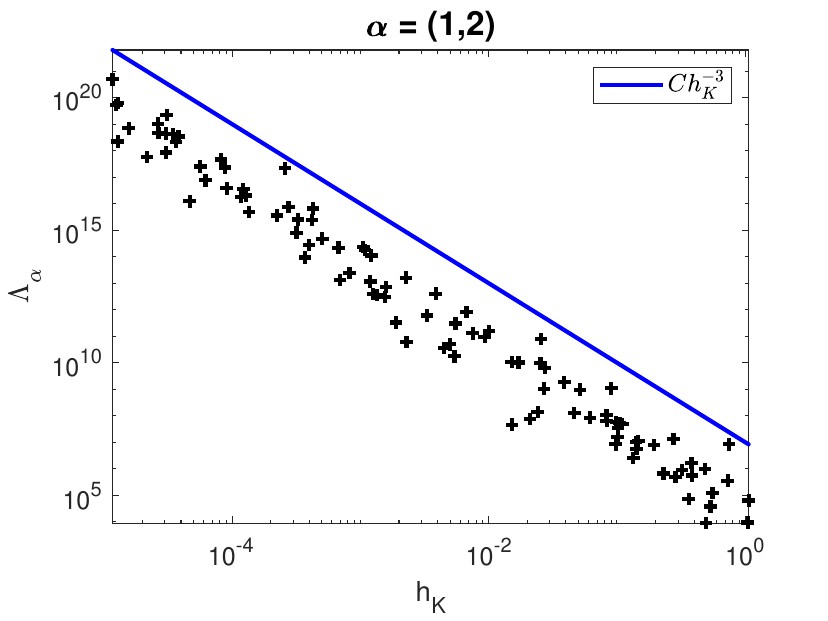}\;
          \includegraphics[width=6cm]{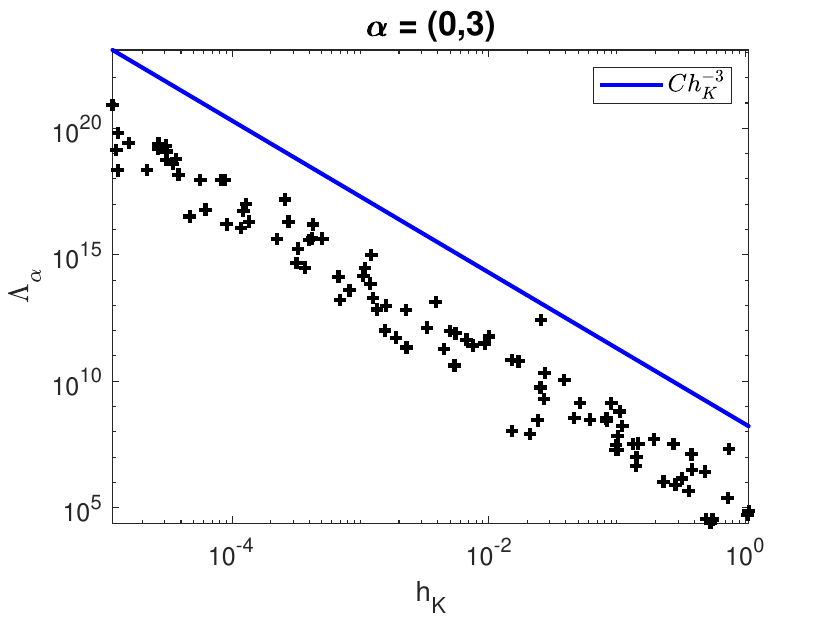}\;
          \caption{Relation of $\Lambda_{\valpha}$ and $h_K$ with $|\valpha| = 3$ on random polygons with $\frac{h_*}{h_K} \geq 0.01$. }\label{fig:example1-alpha=3}
    \end{center}
\end{figure}

By examining the width of the strips containing data points in Figs. \ref{fig:example1-alpha=1}-\ref{fig:example1-alpha=3}, we notice that for fixed $\valpha$ and $h_K$, the values of $\Lambda_{\valpha}$ may still be distributed in a large range
with the ratio between upper and lower bounds up to $10^4$. Indeed, this reflects the difference of the constant $C$ in $\Lambda_{\alpha} = Ch_K^{-|\valpha|}$.
According to Lemma \ref{bounds-of-phi_v} and Theorem \ref{cor:bound}, constant $C$ depends on $|F|$ and $|V|$. In $2$-dimension, one has $|F|=|V|$.
We further design a numerical experiment to test the dependency of $\Lambda_{\valpha}$ on $\nof{V}$.
This time, $100$ random convex polygons with $h_K = O(1)$, $\frac{h_*}{h_K}\ge 0.01$, and $|V|\le 10$ are generated.
We take $|V|\le 10$ because, according to Fig. \ref{fig:CVT-polygon}, polygons with higher number of vertices seldom appear in a practical mesh.
Dependency relation of $\Lambda_{\valpha}$ on $\nof{V}$ are reported in Figs. \ref{fig:nv-Lambda-alpha-eta001-1}-\ref{fig:nv-Lambda-alpha-eta001-3},
for $|\valpha| = 1,2,3$. It appears that $\Lambda_{\valpha}$ increases exponentially as $|V|$ grows. 
However, since $|V|$ stays between $[3,\, 10]$, the range of $\Lambda_{\valpha}$ remains controlled for each given $\valpha$.

\begin{figure}[h]
  \centering
  \includegraphics[width=6cm]{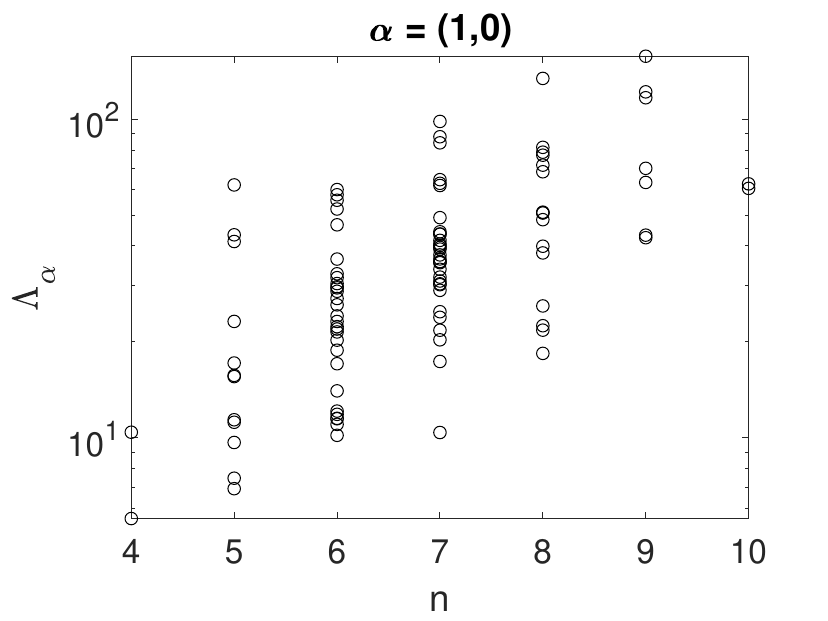}\;
  \includegraphics[width=6cm]{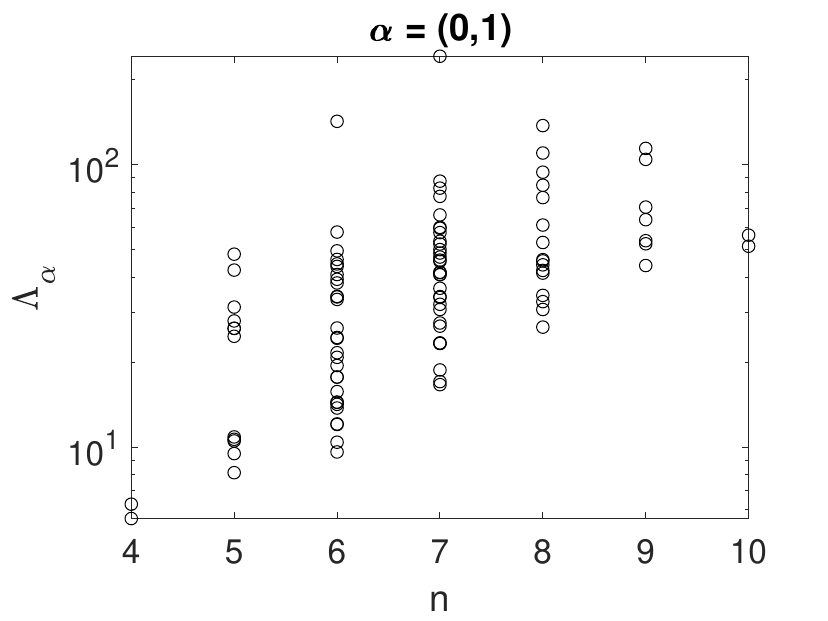}\;
  \caption{Relation of $\Lambda_{\valpha}$ and $\nof{V} = n $ with $|\valpha|= 1$ on random polygons with $\frac{h_*}{h_K} \geq 0.01$ and $h_K=O(1)$.}\label{fig:nv-Lambda-alpha-eta001-1}
\end{figure}

\begin{figure}[h]
  \centering
  \includegraphics[width=6cm]{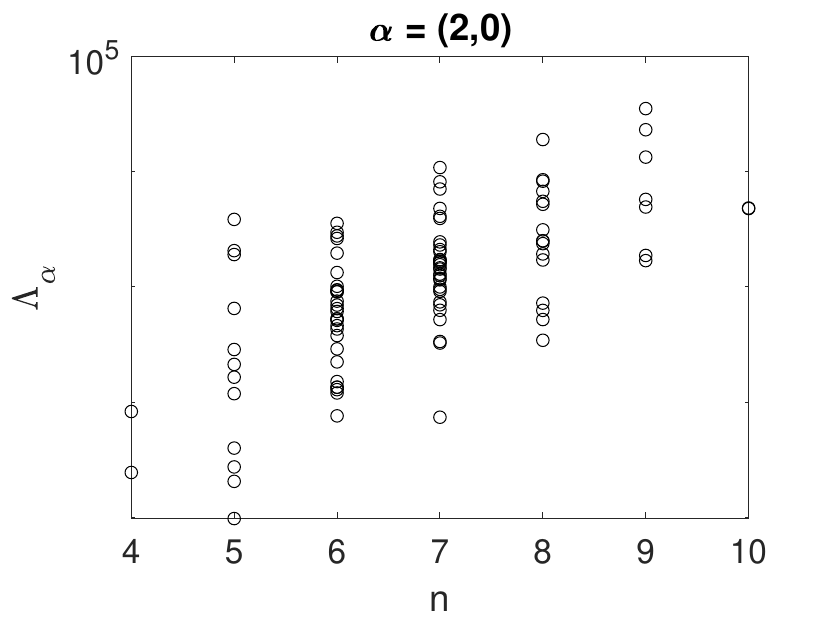}\;
  \includegraphics[width=6cm]{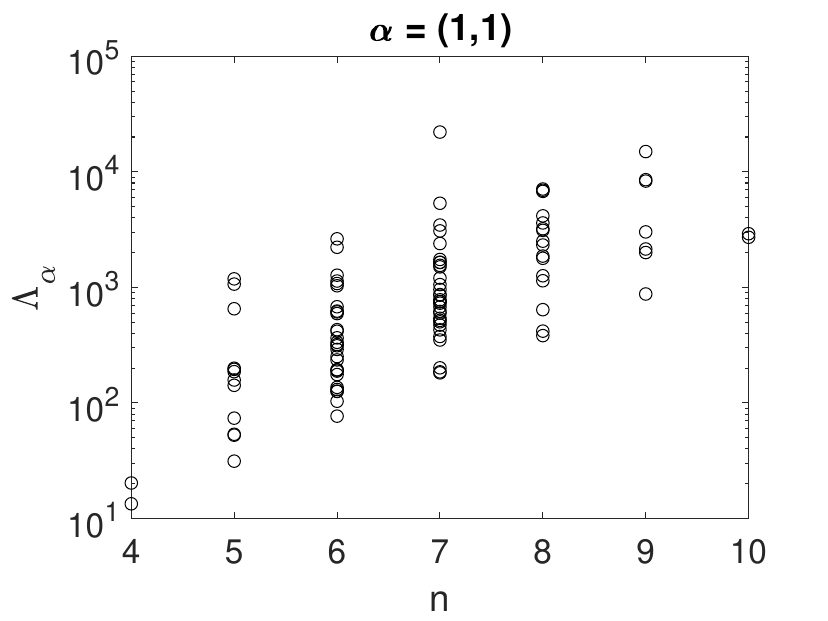}\;
  \includegraphics[width=6cm]{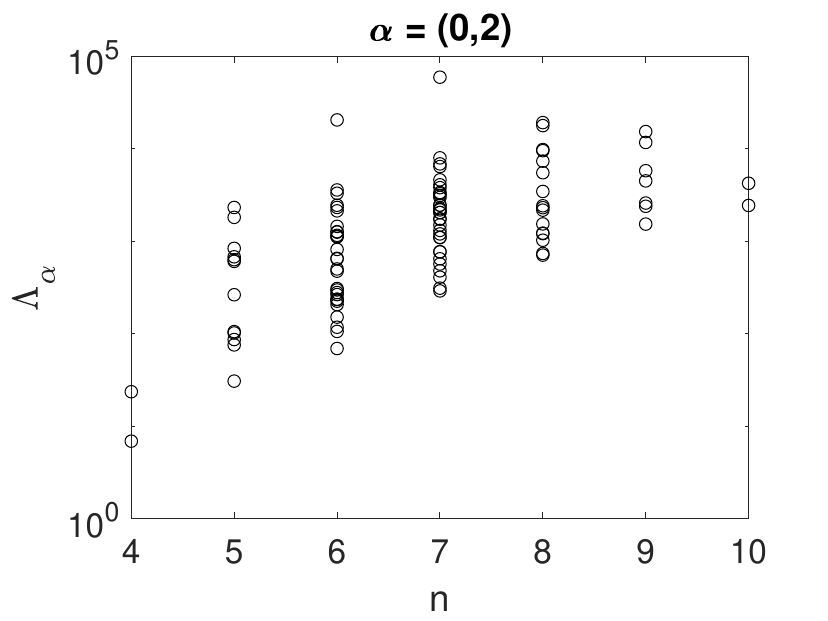}\;
  \caption{Relation of $\Lambda_{\valpha}$ and $\nof{V} = n $ with $|\valpha|= 2$ on random polygons with $\frac{h_*}{h_K} \geq 0.01$ and $h_K=O(1)$.}\label{fig:nv-Lambda-alpha-eta001-2}
\end{figure}

\begin{figure}[h]
  \centering
  \includegraphics[width=6cm]{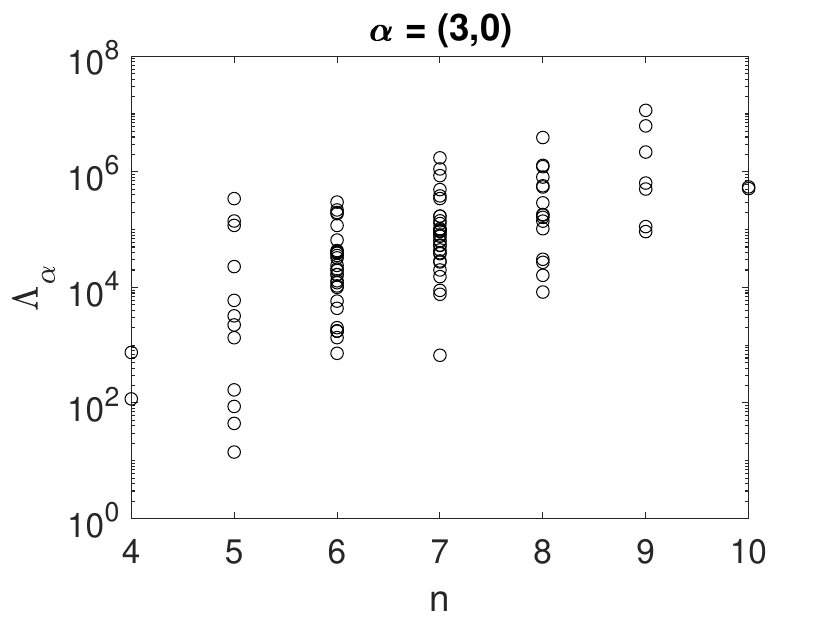}\;
  \includegraphics[width=6cm]{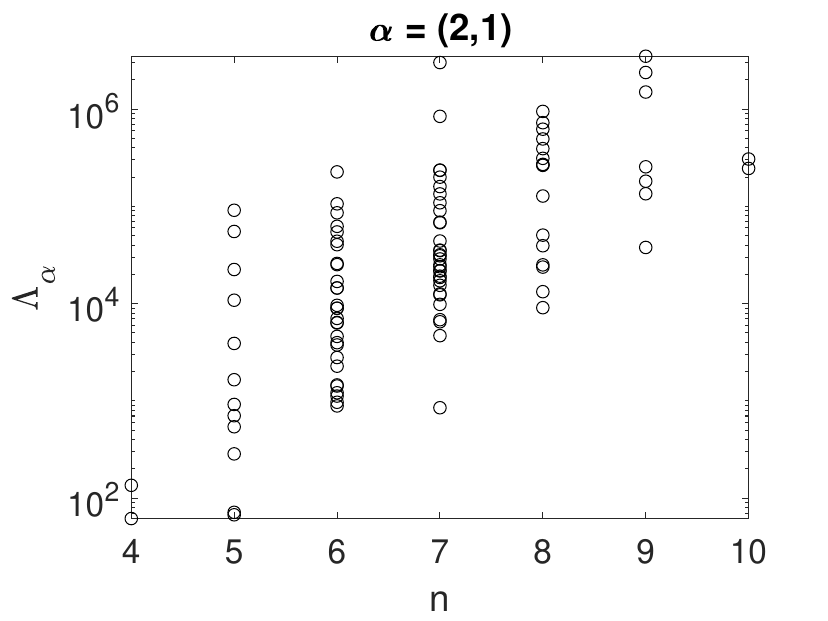}\;
  \includegraphics[width=6cm]{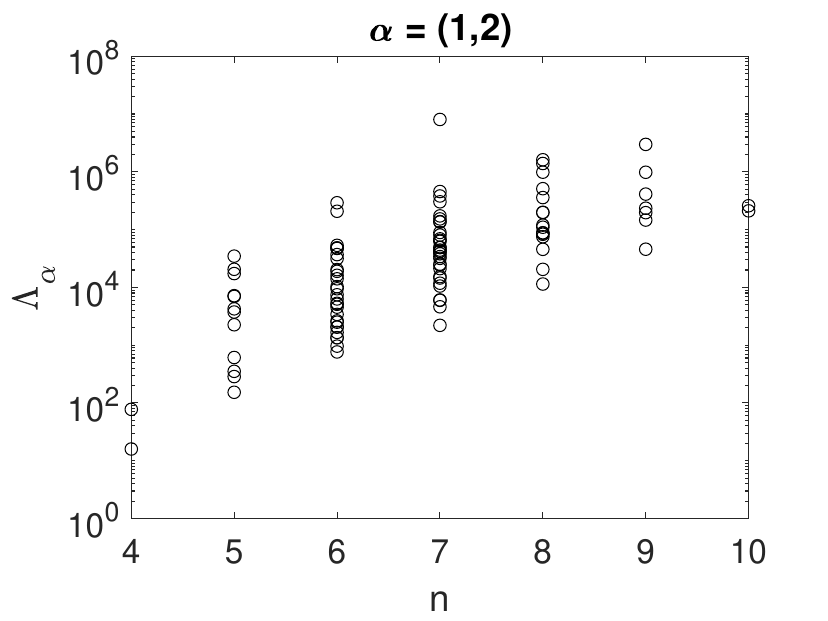}\;
  \includegraphics[width=6cm]{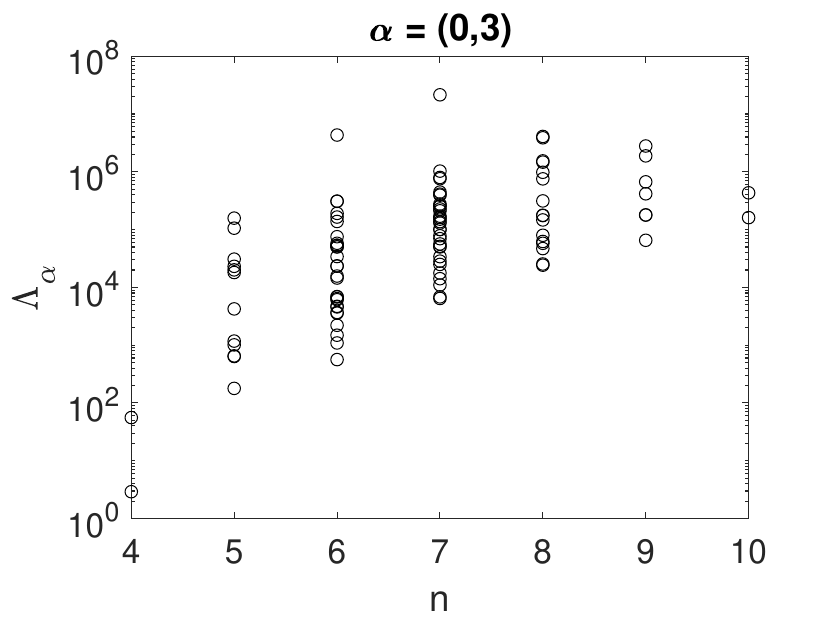}\;
  \caption{Relation of $\Lambda_{\valpha}$ and $\nof{V} = n $ with $|\valpha|= 3$ on random polygons with $\frac{h_*}{h_K} \geq 0.01$ and $h_K=O(1)$.}\label{fig:nv-Lambda-alpha-eta001-3}
\end{figure}


\subsection{What happens when $h_* \ll h_K$? }
We then investigate the influence of $h_*$ on $\Lambda_{\alpha}$. 
To this end, we consider three polygons: $K_1$ with vertices
$$ \vv_1 = (0,0),\; \vv_2 = (a,0),\; \vv_3 = (0.8,0.4), \;
   \vv_4 = (0.7,0.7),\; \vv_5 = (0.1,1),\;$$
$$ \vv_6 = (-0.2,0.5),$$
and $K_2$ with vertices
$$ 
\vv_1 = (0,0),\, \vv_2 = (0,a),\, \vv_3 = (-0.6,0.7),\, \vv_4 = (-1,0.4),\, \vv_5 = (-0.9,0.15),\,
$$
$$ \vv_6 = (-0.24,-0.2), $$
and $K_3$ with vertices
$$ 
\vv_1 = (0,0),\, \vv_2 = (a,a),\, \vv_3 = (0.1,0.7),\, \vv_4 = (-0.2,0.8),\, \vv_5 = (-0.5,0.5),\,
$$
$$ \vv_6 = (-0.5,-0.1), $$
where $0<a<1$.
Illustrations of $K_1$, $K_2$ and $K_3$ are given in Fig. \ref{fig:polygon-K}.
When $a \rightarrow 0$, polygon $K_1$, $K_2$ and $K_3$ has $h_*\rightarrow 0$.
We consider these three polygons as they represent three different cases where $h_*\rightarrow 0$ occurs in the 
$x$-direction, the $y$-direction and $y = x$ direction, respectively. 
It is interesting to examine whether they affect only the partial derivatives with respect to $x$, to $y$, or both.

\begin{figure}[h]
  \centering
  \includegraphics[width=4cm]{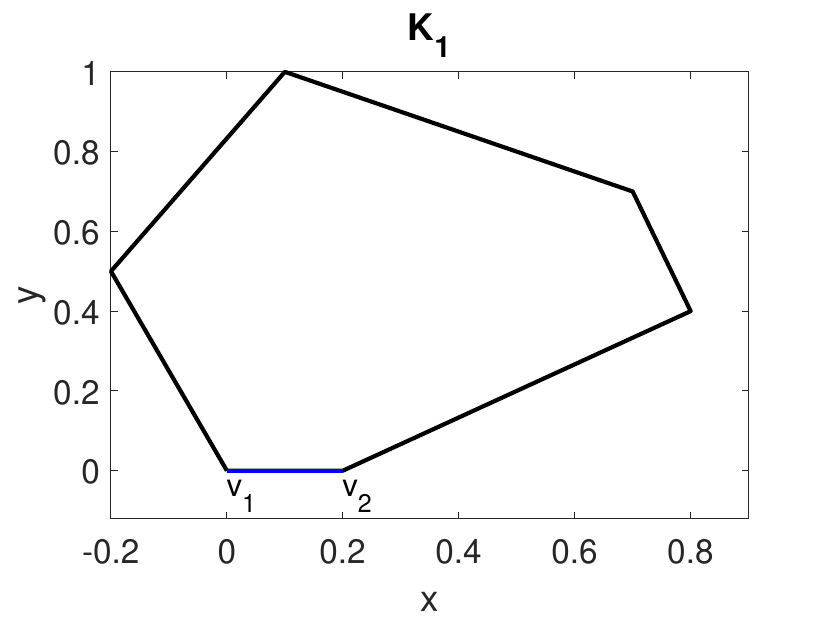}\;
  \includegraphics[width=4cm]{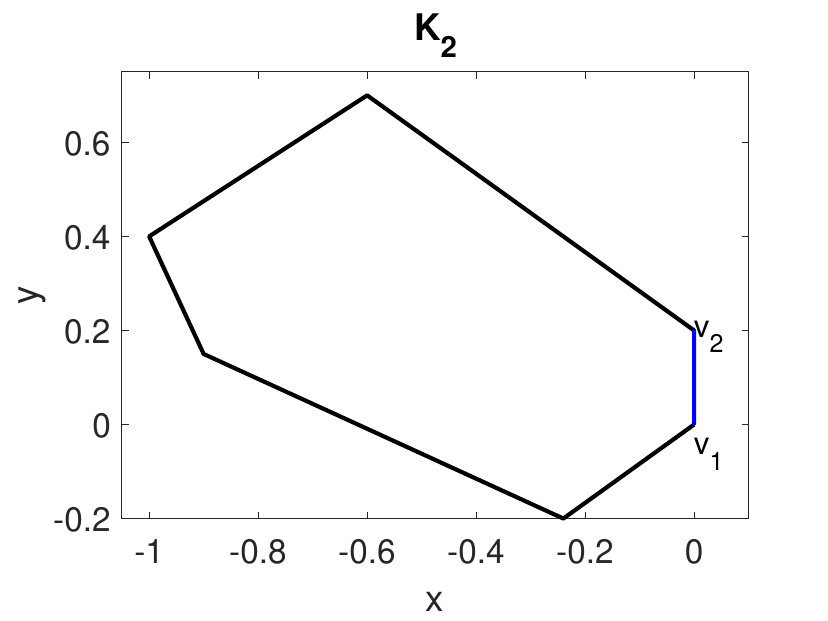}\;  \includegraphics[width=4cm]{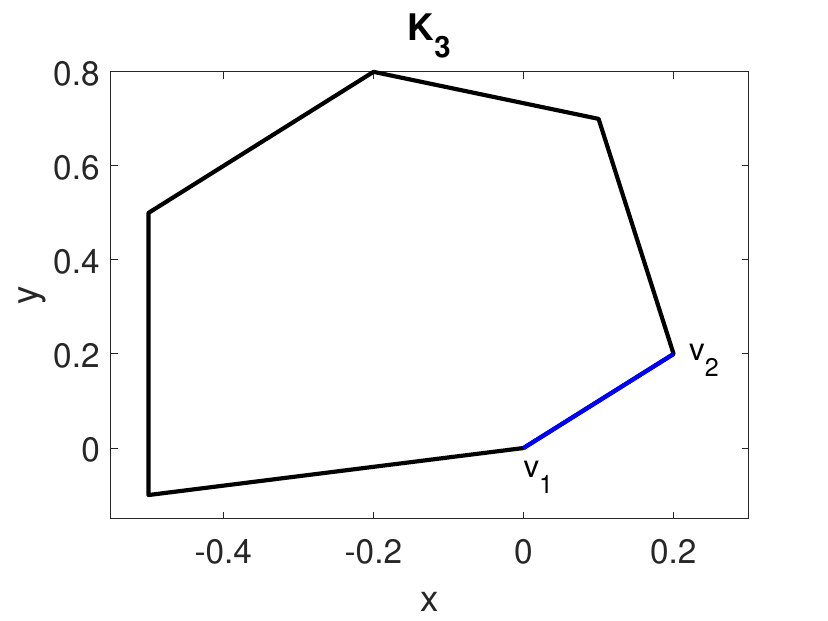}\;
  \caption{$K_1$, $K_2$ and $K_3$.}\label{fig:polygon-K}
\end{figure}

In Figs. \ref{fig:Lambda-alpha-K1}-\ref{fig:Lambda-alpha-K3}, we report $\Lambda_{\valpha}$ vs. $h_*$ in $K_1$, $K_2$ and $K_3$.
Because $h_K=O(1)$, in all graphs we observe that $\Lambda_{\valpha}\le Ch_*^{|\valpha|}$,
some having a slower increasing rate than $Ch_*^{|\valpha|}$.
The numerical results immediately tell us that the theoretical result in Lemma \ref{bounds-of-phi_v} is far from sharp.
Based on the numerical results, we suspect that a sharp upper bound for the case $h_K=O(1)$ is $Ch_*^{|\valpha|}$,
which awaits future research.

\begin{figure}[h]
  \centering
  \includegraphics[width=6cm]{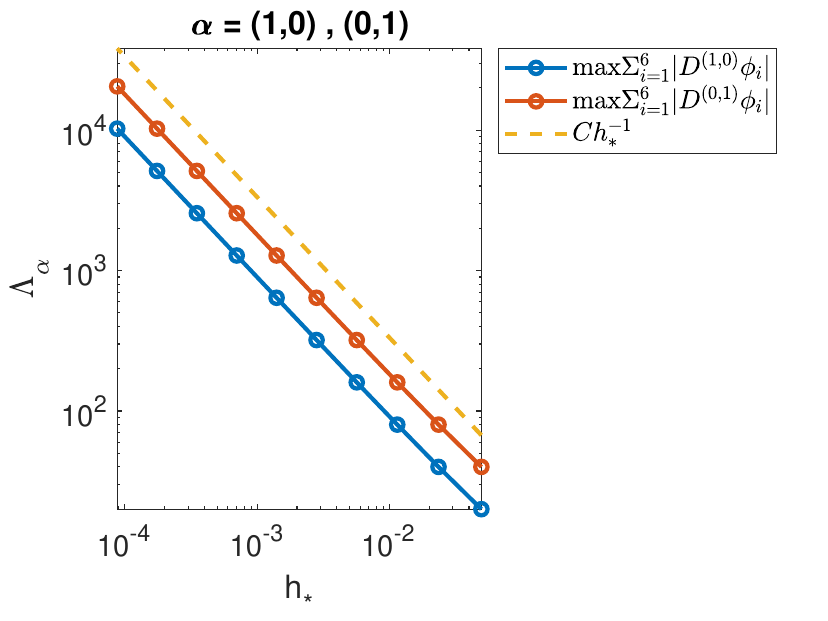}\;
  \includegraphics[width=6cm]{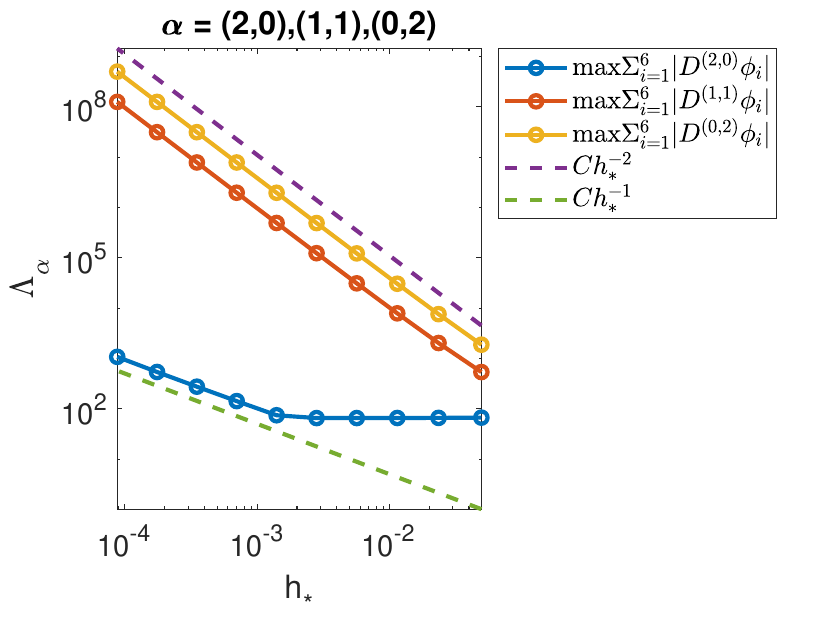}\;
  \includegraphics[width=6cm]{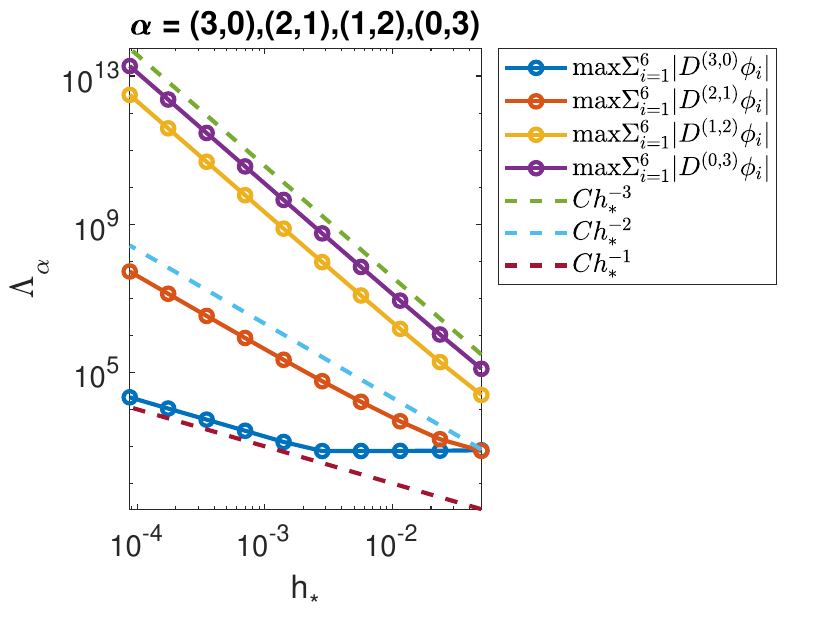}\;
  \caption{From the first to the last panel, relation between $\Lambda_{\valpha}$ and $h_*$ with $|\valpha|=1,2,3$ on $K_1$.}\label{fig:Lambda-alpha-K1}
\end{figure}

\begin{figure}[h]
	\centering
	\includegraphics[width=6cm]{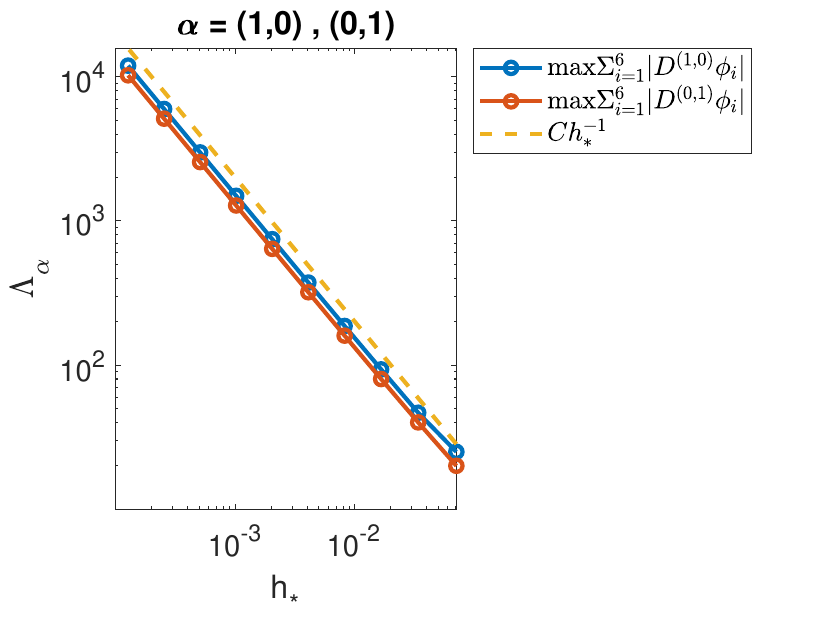}\;
	\includegraphics[width=6cm]{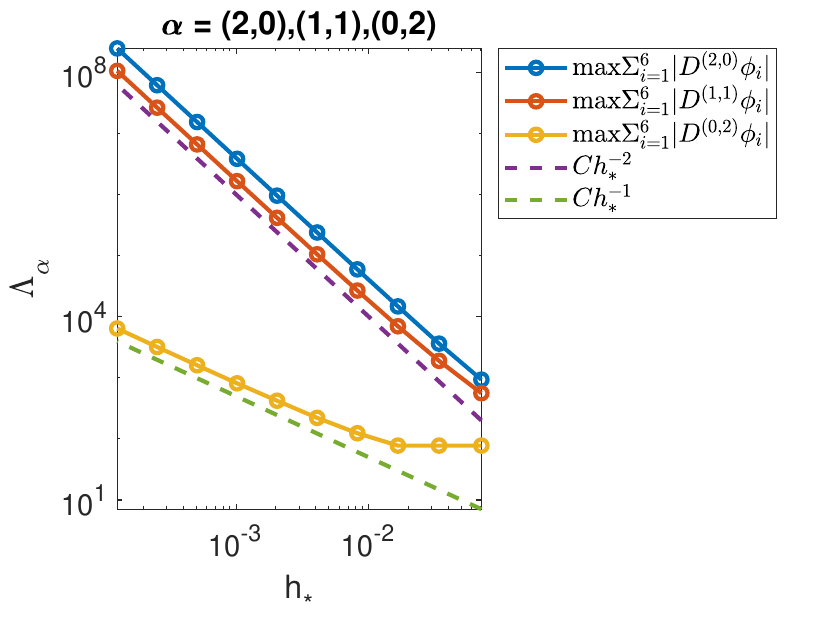}\;
	\includegraphics[width=6cm]{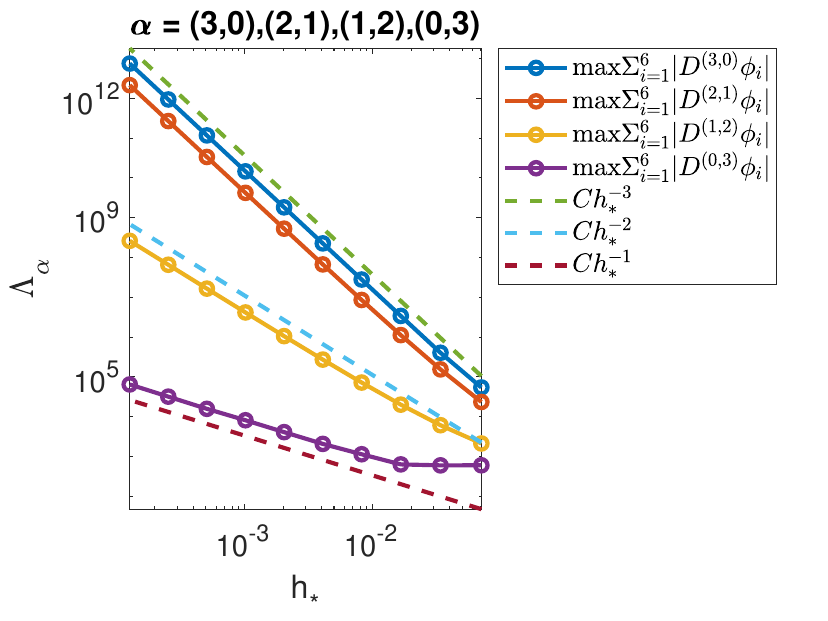}\;
	\caption{From the first to the last panel, relation between $\Lambda_{\valpha}$ and $h_*$ with $|\valpha|=1,2,3$ on $K_2$.}\label{fig:Lambda-alpha-K2}
\end{figure}

\begin{figure}[h]
	\centering
	\includegraphics[width=6cm]{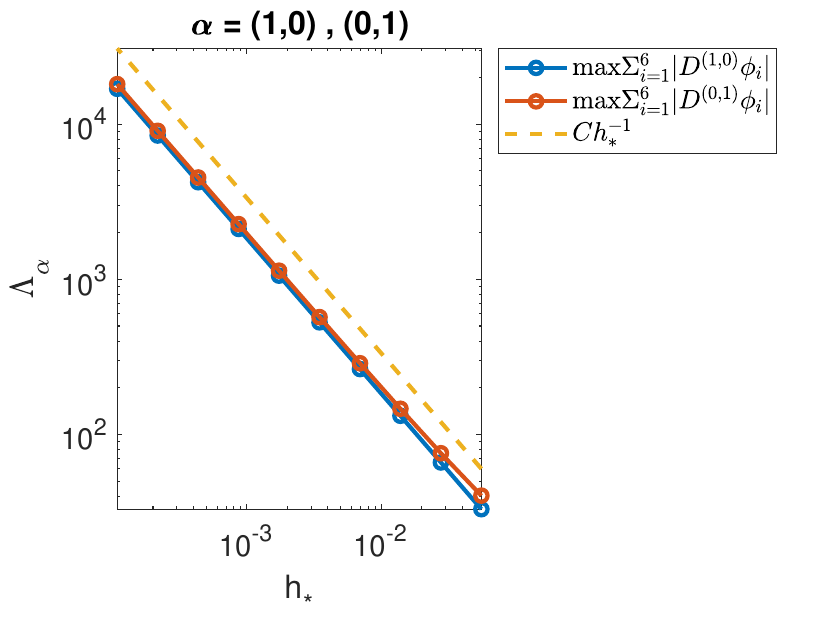}\;
	\includegraphics[width=6cm]{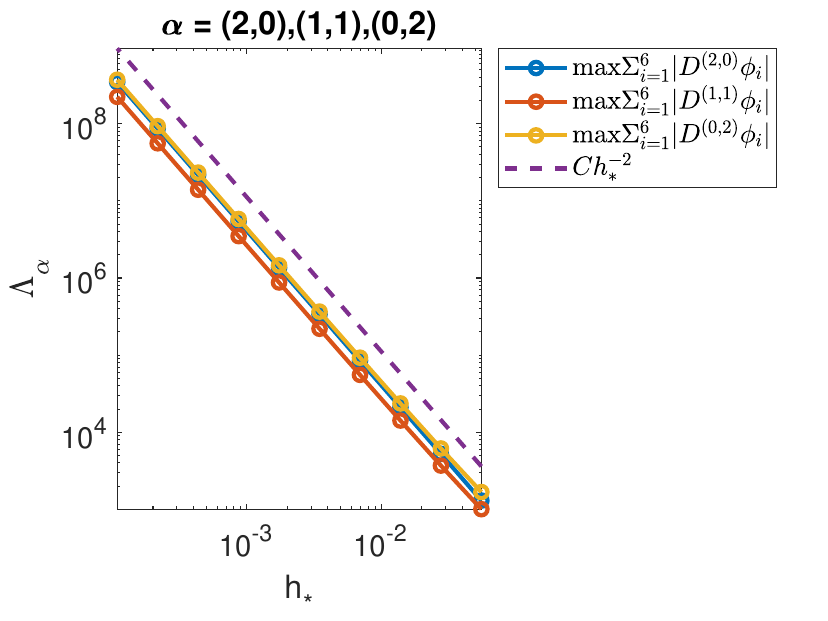}\;
	\includegraphics[width=6cm]{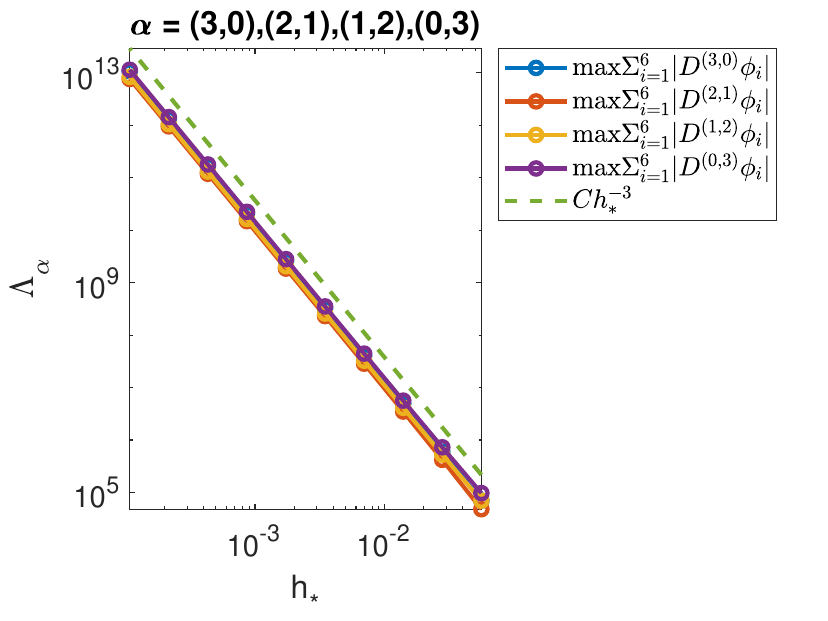}\;
	\caption{From the first to the last panel, relation between $\Lambda_{\valpha}$ and $h_*$ with $|\valpha|=1,2,3$ on $K_3$.}\label{fig:Lambda-alpha-K3}
\end{figure}

Comparing Figs. \ref{fig:Lambda-alpha-K1}-\ref{fig:Lambda-alpha-K3}, one easily see that $\Lambda_{\valpha}$ for $\valpha=(2,0),\,(2,1),\,(3,0)$ on $K_1$,
and for $\valpha=(0,2),\,(1,2),\,(0,3)$ on $K_2$ do no reach the order $Ch_*^{|\valpha|}$ as others do.
We suspect that these derivatives are affected by the direction of the short edges.
To further explore this,
we plot $\max_{\vx\in K} \left|D^{\valpha}\phi_i(\vx)\right|$ for individual $\vv_i$s on polygons $K_1$-$K_3$.
The results are surprisingly rich for various combinations of $\valpha$ and $i$.

\begin{enumerate}
\item We first check $K_3$, which has a slant short edge. In Fig. \ref{fig:single-dphi-K3}, we plot $\max_{\vx\in K} \left|D^{\valpha}\phi_i(\vx)\right|$ on $K_3$,
for $\valpha = (1,0),\,(2,0),\,(3,0)$, respectively.
From the graphs, we observe that $\max_{\vx\in K}|D^{\valpha}\phi_i|$ reaches the asymptotic order $Ch_*^{|\valpha|}$ only when $\vv_i$ is incident to the short edge.
Otherwise, various downgrading in the asymptotic orders can be observed.

Graphs for other values of $\valpha$ on $K_3$ are omitted, since the results for $\valpha$s with the same length behave exactly the same.
That is, the graph on $K_3$ with $|\valpha|=1,\,2,\,3$ looks the same as the graph for $\valpha = (1,0),\,(2,0),\,(3,0)$ in Fig. \ref{fig:single-dphi-K3}, respectively.

\item We then check $K_1$, which has a horizontal short edge. It turns out that except for $\valpha=(2,0),\,(2,1),\,(3,0)$, the results for all other values of $\valpha$
behave the same as in the case of $K_3$. 
In Fig. \ref{fig:single-dphi-K1}, we plot $\max_{\vx\in K} \left|D^{\valpha}\phi_i(\vx)\right|$ on $K_1$, for $\valpha = (2,0),\,(3,0)$.
The left panels show $\max_{\vx\in K}|D^{\valpha}\phi_i|$ for $\vv_i$ incident to the short edge, i.e., $i=1,\,2$,
while the right panels show $\max_{\vx\in K}|D^{\valpha}\phi_i|$ for $\vv_i$ non-incident to the short edge, i.e., $i=3,\, 4,\, 5, \, 6$. 
In both figures, values for $\phi_1$ and $\phi_2$ are highly identical so that the blue lines for $\phi_1$ are hiding behind the brown lines for $\phi_2$.
For $\vv_1$ and $\vv_2$ that are incident to the short edge, the values appear to have asymptotic order $Ch_*^{-1}$, regardless of $\valpha$. 
For vertices that are non-incident to the short edge, the values appear to be independent of $h_*$. 
In Fig. \ref{fig:single-dphi-K1-21}, we plot $\max_{\vx\in K} \left|D^{\valpha}\phi_i(\vx)\right|$ on $K_1$, for $\valpha = (2,1)$.

\item In Fig. \ref{fig:single-dphi-K2}, we plot $\max_{\vx\in K} \left|D^{\valpha}\phi_i(\vx)\right|$ on $K_2$, which has a vertical short edge, for $\valpha = (0,2),\,(0,3)$.
The results are similar to Fig. \ref{fig:single-dphi-K1}.
In Fig. \ref{fig:single-dphi-K2-12}, we plot $\max_{\vx\in K} \left|D^{\valpha}\phi_i(\vx)\right|$ on $K_2$, for $\valpha = (1,2)$.
The results are similar to Fig. \ref{fig:single-dphi-K1-21}.
Again, results for other values of $\valpha$ are not plotted as they behave the same to Fig. \ref{fig:single-dphi-K3}.

\end{enumerate}

\begin{figure}[h]
	\centering
	\includegraphics[width=6cm]{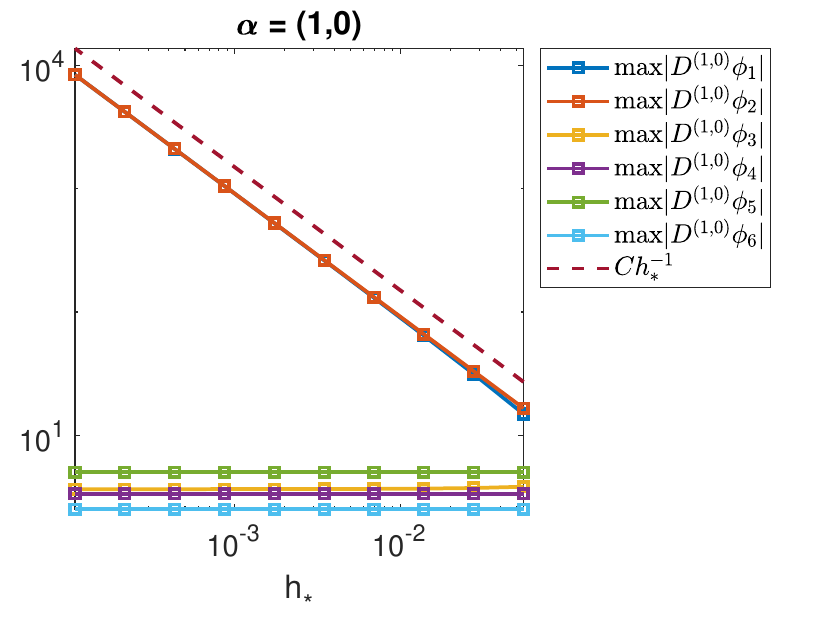}\;
	\includegraphics[width=6cm]{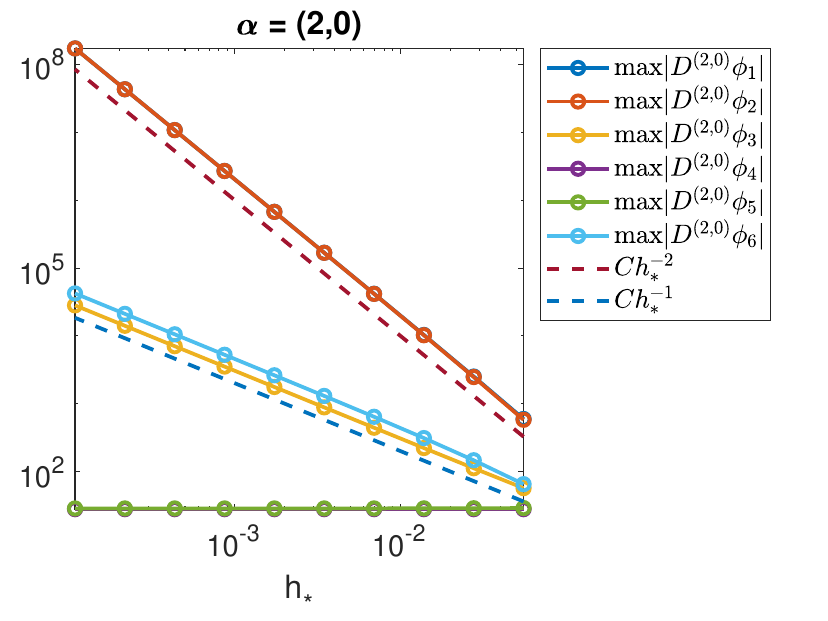}\;\\
	\includegraphics[width=6cm]{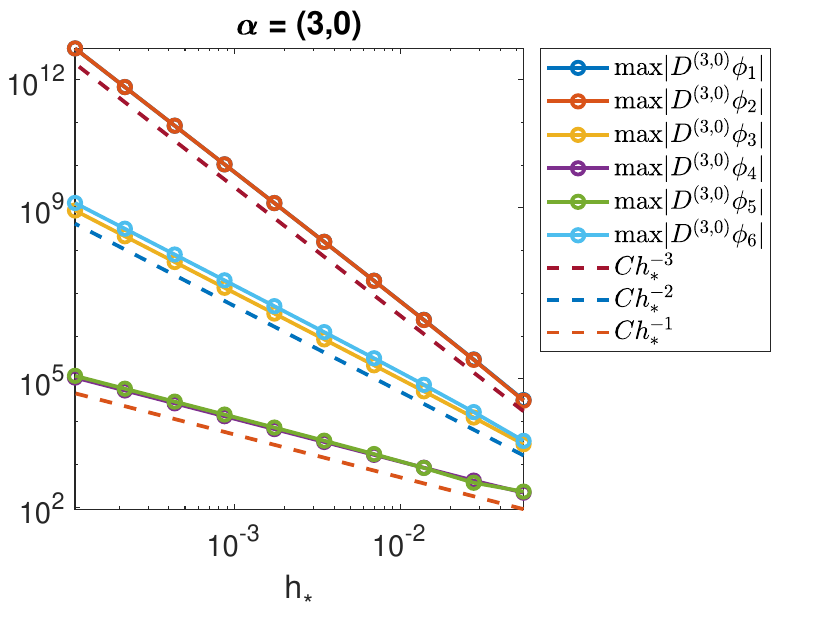}
	\caption{Relation between $\max_{\vx\in K}|D^{\valpha}\phi_i(\vx)|$ and $h_*$ for $\valpha=(1,0),\,(2,0),\,(3,0)$ on $K_3$.}\label{fig:single-dphi-K3}
\end{figure}

%

\begin{figure}[h]
	\centering
	\includegraphics[width=6cm]{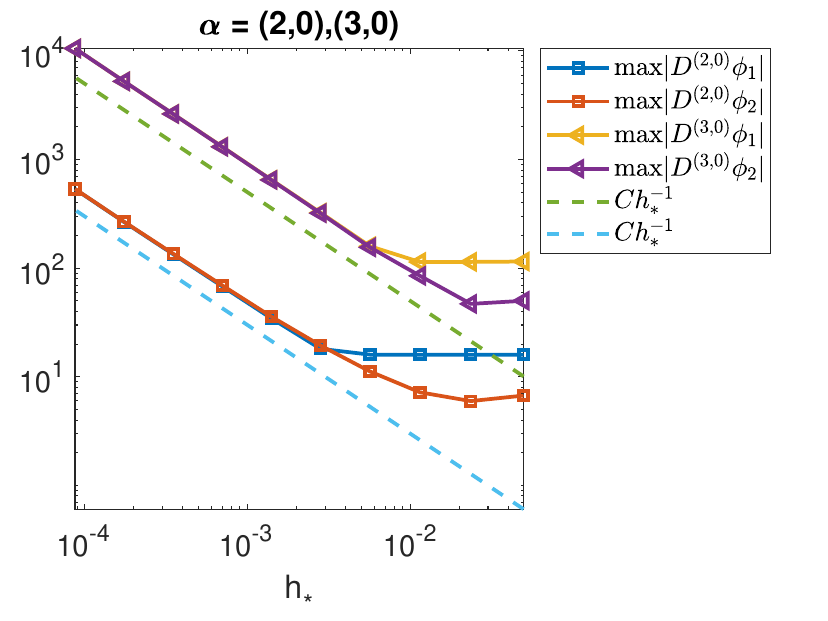}\;
	\includegraphics[width=6cm]{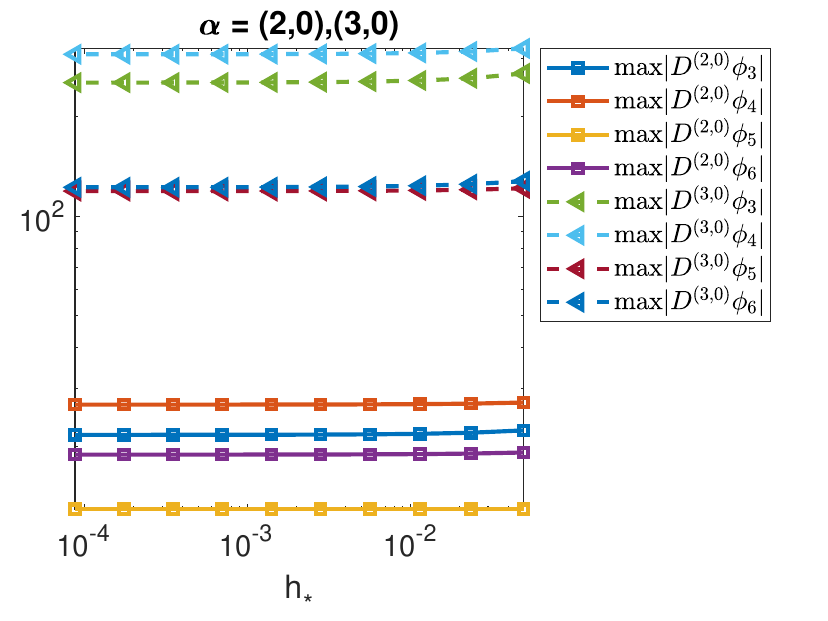}\;
	\caption{Relation between $\max_{\vx\in K}|D^{\valpha}\phi_i(\vx)|$ and $h_*$ for $\valpha=(2,0)$ and $(3,0)$ on $K_1$: left panel for $i= 1,\,2$, right panel for $i= 3,\,4,\,5,\,6$.}\label{fig:single-dphi-K1}
\end{figure}

\begin{figure}[h]
	\centering
	\includegraphics[width=6cm]{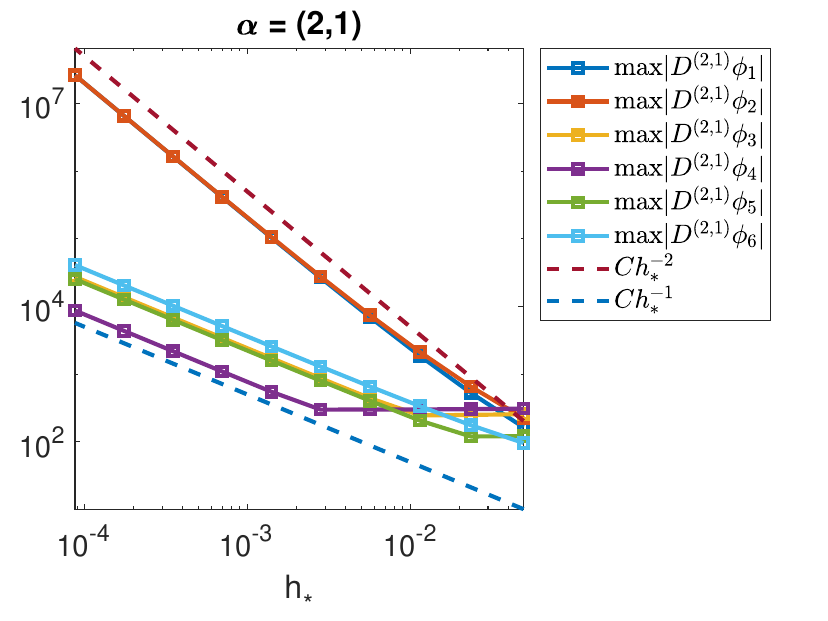}
	\caption{Relation between $\max_{\vx\in K}|D^{\valpha}\phi_i(\vx)|$ and $h_*$ for $\valpha=(2,1)$ on $K_1$.}\label{fig:single-dphi-K1-21}
\end{figure}

\begin{figure}[h]
	\centering
	\includegraphics[width=6cm]{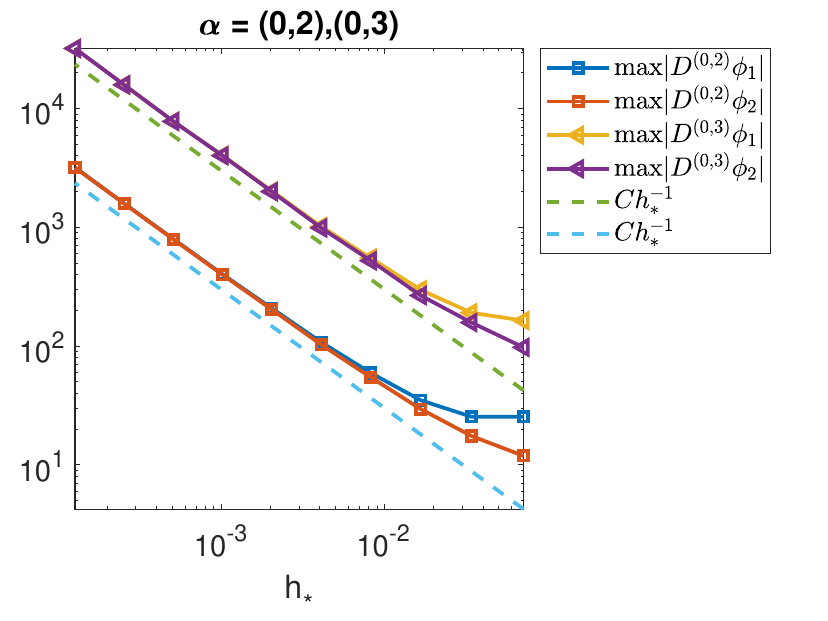}\;
	\includegraphics[width=6cm]{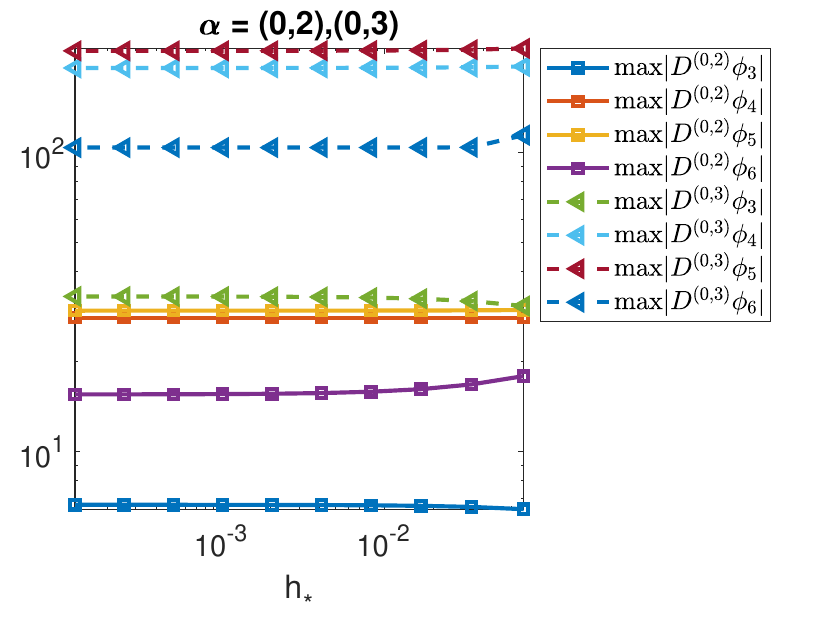}\;
	\caption{Relation between $\max_{\vx\in K}|D^{\valpha}\phi_i(\vx)|$ and $h_*$ for $\valpha=(0,2)$ and $(0,3)$ on $K_2$: left panel for $i= 1,\,2$, right panel for $i= 3,\,4,\,5,\,6$.}\label{fig:single-dphi-K2}
\end{figure}

\begin{figure}[h]
	\centering
	\includegraphics[width=6cm]{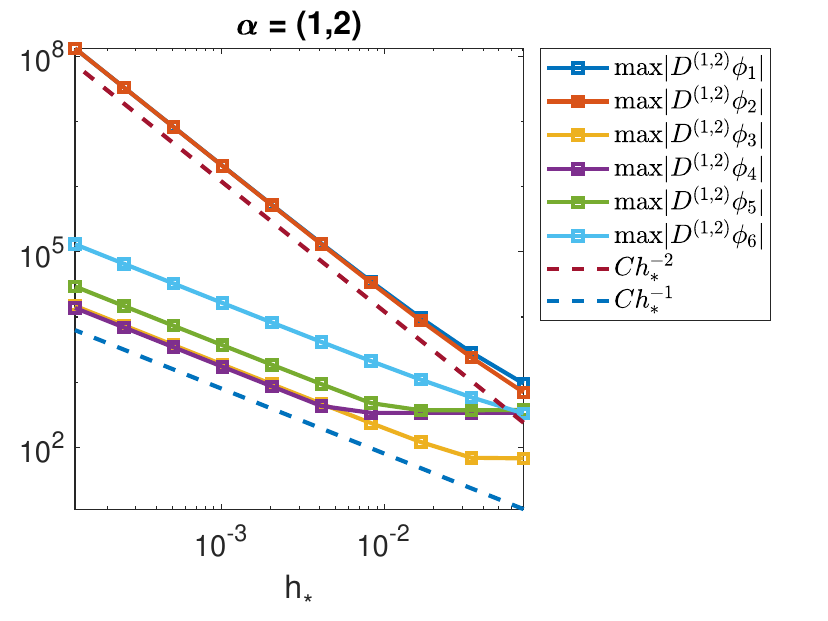}
	\caption{Relation between $\max_{\vx\in K}|D^{\valpha}\phi_i(\vx)|$ and $h_*$ for $\valpha=(1,2)$ on $K_2$.}\label{fig:single-dphi-K2-12}
\end{figure}

Based on the above observation, we suspect that the directional derivative, in the tangential direction of the short edge, may have a differential behavior than other directions.
To this end,  we plot these directional derivatives on $K_3$ in Fig. \ref{fig:single-dphi-K3-tangent}. 
The results appear to be the same as Fig. \ref{fig:single-dphi-K1} and Fig. \ref{fig:single-dphi-K2}, which seem to confirm our conclusion.

\begin{figure}[h]
	\centering
	\includegraphics[width=6cm]{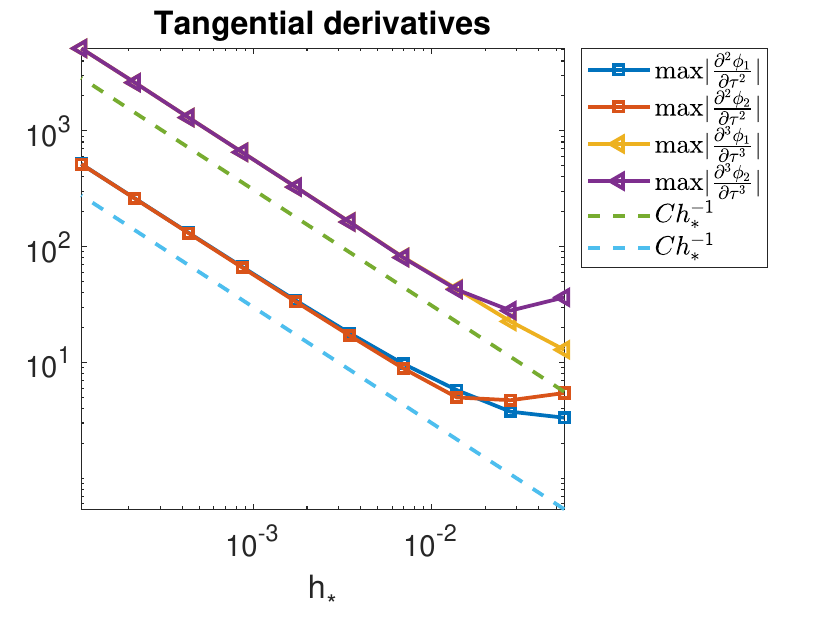}\;
	\includegraphics[width=6cm]{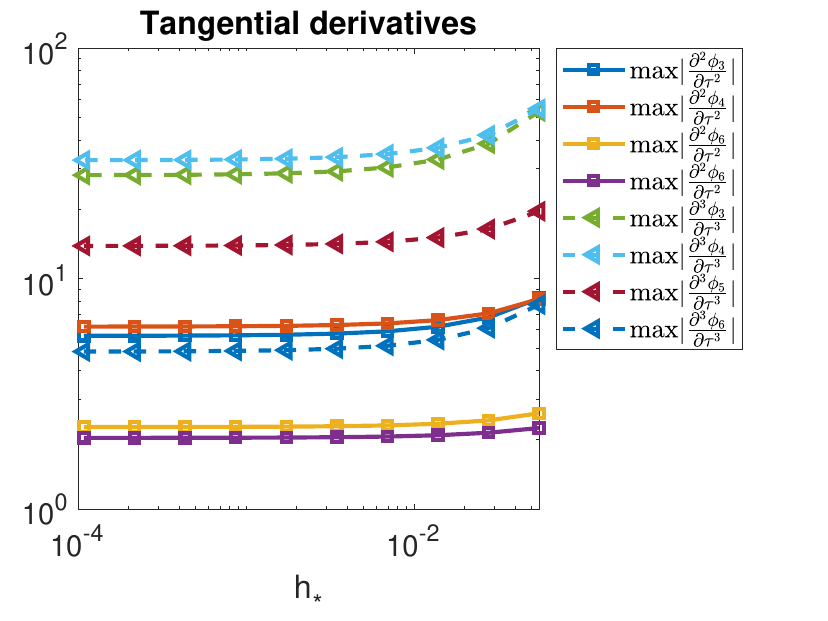}\;
	\caption{Relation between $\max_{\vx\in K}|\frac{\partial^{2}\phi_i(\vx)}{\partial\tau^2}|$, $\max_{\vx\in K}|\frac{\partial^{3}\phi_i(\vx)}{\partial\tau^3}|$ and $h_*$,
	with $\tau$ being the tangent direction of the short edge, on $K_3$: left panel for $i=1,\,2$, right panel for $i= 3,\,4,\,5,\,6$.}\label{fig:single-dphi-K3-tangent}
\end{figure}

In many graphs in Figs. \ref{fig:Lambda-alpha-K1}-\ref{fig:single-dphi-K3-tangent}, some lines are not straight. For now, we are not able to explain this.

\subsection{When does one need a higher ($|\valpha|>1$) derivative bound?} \label{subsec:FEM}
The present research was motivated by the application of GBC-based FEM to fourth-order elliptic equations.
Let $\Omega = (0,1)\times (0,1)$. 
Consider the clamped plate bending problem: Find $u \in H_0^2(\Omega)$ such that for all $v \in  H_0^2(\Omega)$,
\begin{equation}\label{weakform}
	\int_{\Omega} \left(\Delta u \Delta v + (1 - \sigma )(2\frac{\partial^2u}{\partial x\partial y}\frac{\partial^2v}{\partial x\partial y} - \frac{\partial^2u}{\partial x^2}\frac{\partial^2v}{\partial y^2} - \frac{\partial^2u}{\partial y^2}\frac{\partial^2v}{\partial x^2})\right)\mathrm{d}\vx = \int_{\Omega} fv\, \mathrm{d}\vx, 
\end{equation}
where $f \in L^2(\Omega)$ is a transverse force and $0<\sigma<\frac{1}{2}$ is the Poisson ratio.
One may also consider the simply supported plate model, with a little modification on the boundary conditions.
In \cite{tian2022}, we proposed a Morley-type, GBC-based polygonal FEM for Problem \eqref{weakform}.

We use the second order GBC-based element constructed by Floater and Lai \cite{LaiFloater2016} to define the local finite element space,
while the degrees of freedom are Morley-type, i.e., consisting of the nodal value at each vertex and the momentum of normal derivative on each edge.
The element constructed in \cite{tian2022} is a natural extension of the Morley element \cite{morley1968} to polygonal meshes.
To analyze its FEM approximation error, one needs the upper bound of $ \nabla^2\phi_i$.

In \cite{tian2022}, we were not able to prove this upper bound at that time, but making it an assumption. 
Under this assumption, 
optimal FEM approximation error $|u-u_h|_{2,\Omega}=O(h)$ and $|u-u_h|_{1,\Omega}= O(h^2)$ are proved in \cite{tian2022}.
The present work fills this missing theoretical step.

 For reader's convenience, we attach below some numerical results from \cite{tian2022}.
Consider the clamped plate with exact solution $u(x,y) = x^2y^2(1-x)^2(1-y)^2$, $u|_{\partial\Omega} = \frac{\partial u}{\partial\vn}|_{\partial\Omega} \equiv 0$, 
and the simply supported plate with exact solution $u(x,y) = e^x\sin (2\pi x)\sin(2\pi y)$,  $u|_{\partial\Omega} \equiv 0, \frac{\partial u}{\partial\vn}|_{\partial\Omega} \not\equiv 0$.
We test the problems on five different types of meshes, as shown in Fig. \ref{fig:meshes}. 
Numerical results are reported in Fig. \ref{fig:allmeshu}-\ref{fig:allmeshU}, where
one observes an $O(h)$ convergence in the $H^2$-seminorm and an $O(h^2)$ convergence in the $H^1$-seminorm.
Both agree well with the theoretical prediction.

\begin{figure}[h]
	\begin{center}
		\includegraphics[width=3.8cm]{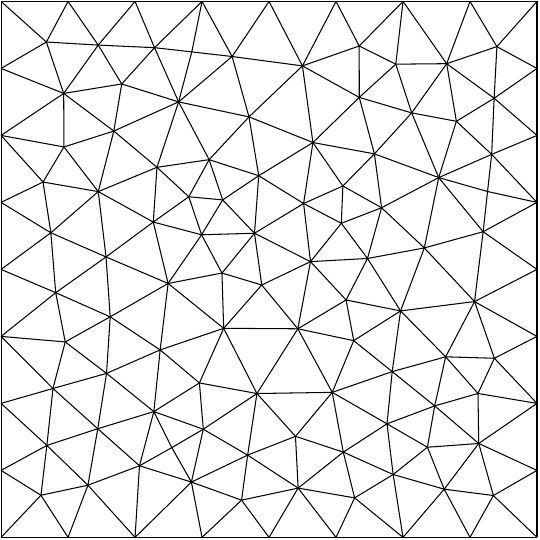}\;
		\includegraphics[width=3.8cm]{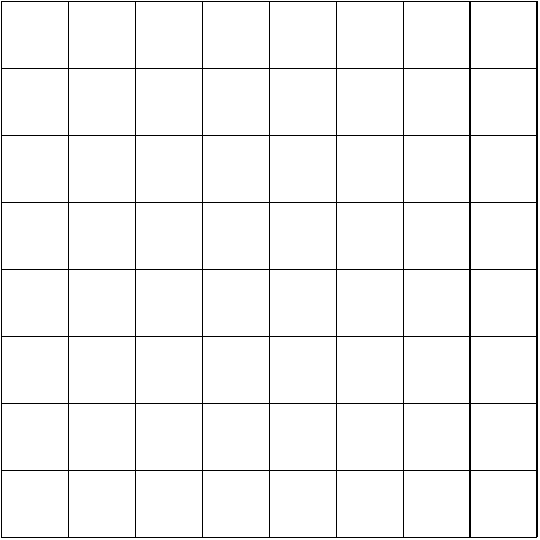}\;
		\includegraphics[width=3.8cm]{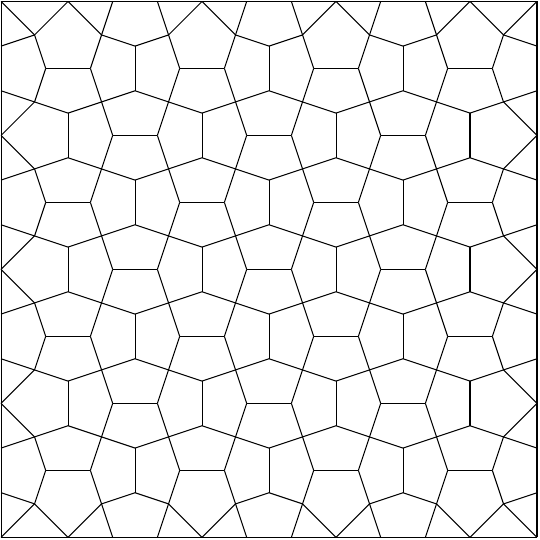}\;
		\includegraphics[width=3.8cm]{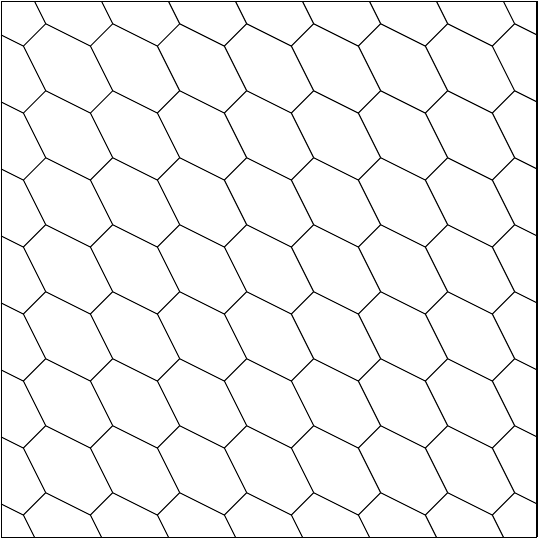}\;
		\includegraphics[width=3.8cm]{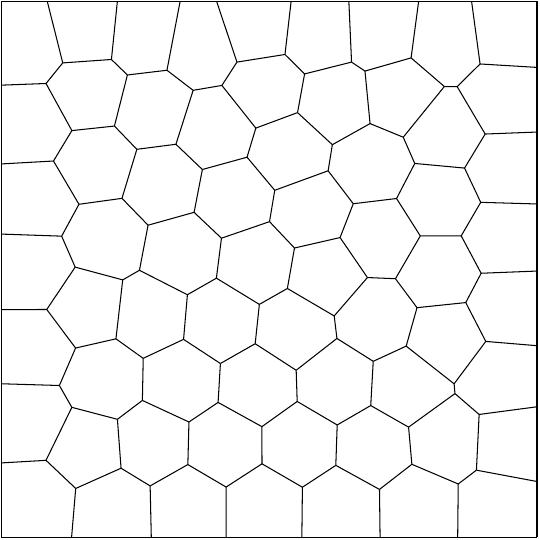}\;
	\end{center}
	\caption{Meshes of size $8\times 8$. (1)\ A triangle mesh. (2)\ A rectangular mesh. (3)\ A pentagons mesh, with mostly pentagons and a few triangles. (4)\ A hexagonal mesh, with mostly hexagons and a few pentagons and quadrilaterals. (5)\ Centroidal Voronoi tessellation.} \label{fig:meshes}
\end{figure}

\begin{figure}[h]
	\begin{center}
		\includegraphics[width=5cm]{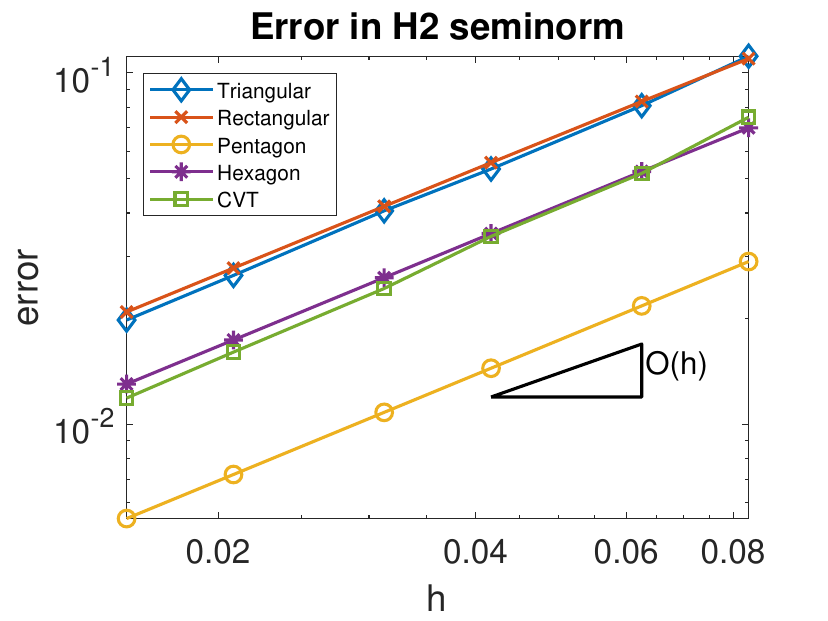}\;
		\includegraphics[width=5cm]{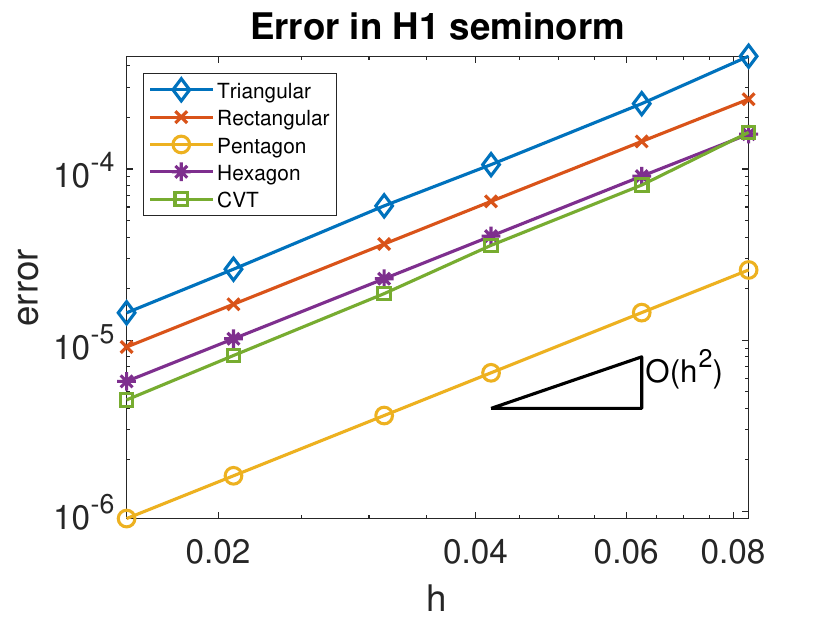}\;
	\end{center}
	\caption{$|u-u_h|_{2,\Omega}$ and $|u-u_h|_{1,\Omega}$ of Morley-type GBC-based FEM for clamped plate on five types of meshes in Fig. \ref{fig:meshes}.}\label{fig:allmeshu}
\end{figure}

\begin{figure}[h]
	\begin{center}
		\includegraphics[width=5cm]{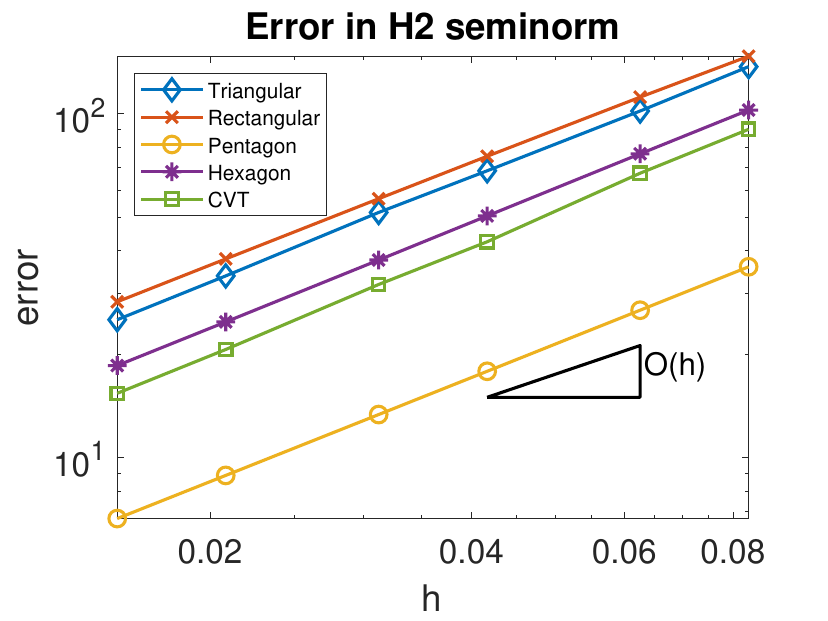}\;
		\includegraphics[width=5cm]{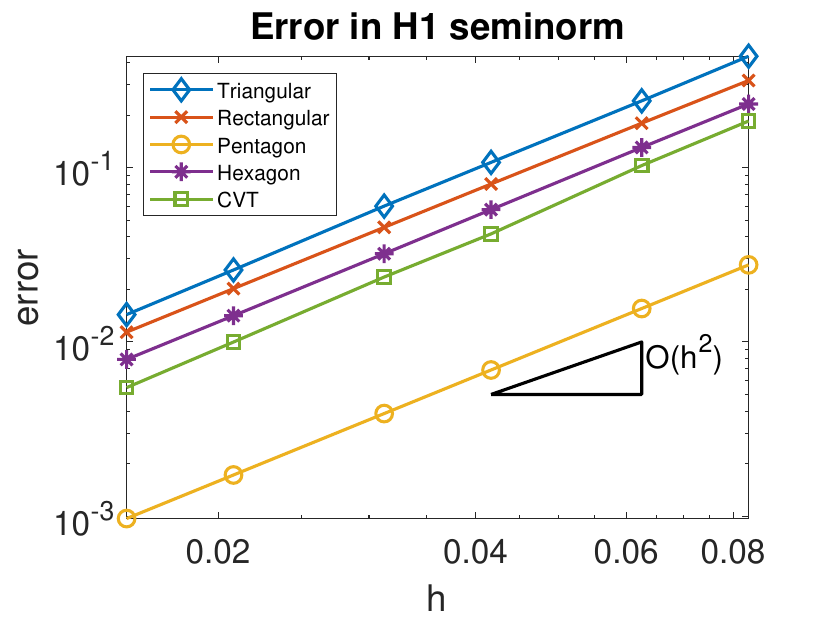}\;
	\end{center}
	\caption{$|u-u_h|_{2,\Omega}$ and $|u-u_h|_{1,\Omega}$ of Morley-type GBC-based FEM for simply supported plate on five types of meshes in Fig. \ref{fig:meshes}.}\label{fig:allmeshU}
\end{figure}

\end{document}